\documentclass[a4paper]{scrartcl}
\usepackage[margin=1in]{geometry}

\usepackage[utf8]{inputenc}
\usepackage[T1]{fontenc}
\usepackage{lmodern}
\usepackage{microtype}

\usepackage[dvipsnames]{xcolor}
\usepackage{subcaption}
\usepackage{booktabs}
\usepackage{tikz}

\usepackage{amsmath}
\usepackage{amssymb}
\usepackage{amsthm}
\newtheorem{theorem}{Theorem}
\newtheorem{lemma}[theorem]{Lemma}
\newtheorem{corollary}[theorem]{Corollary}
\theoremstyle{definition}
\newtheorem{remark}[theorem]{Remark}
\newtheorem{example}[theorem]{Example}

\usepackage{enumitem}
\setlist{noitemsep,topsep=0pt}

\usepackage[backend=biber,style=alphabetic,maxnames=10]{biblatex}
\usepackage{breakurl}
\usepackage[colorlinks]{hyperref}
\hypersetup{
  linkcolor=BrickRed,
  citecolor=Green,
  urlcolor=NavyBlue,
  menucolor=BrickRed
}

\newcommand{\N}{\mathbb{N}}

\newcommand{\R}{\mathbb{R}}
\newcommand{\Rge}{\R_{\ge 0}}
\newcommand{\Rle}{\R_{\le 0}}
\newcommand{\bigO}{\mathcal{O}}
\newcommand{\set}[1]{\{ #1 \}}
\newcommand{\fromto}[2]{\{ #1, \dots, #2 \}}

\DeclareMathOperator{\st}{\,s.t.}
\DeclareMathOperator*{\argmin}{argmin}
\DeclareMathOperator{\ratio}{ratio}

\usepackage{authblk}

\usepackage{orcidlink}

\begin{document}

\title{The Complexity Landscape of Two-Stage Robust Selection Problems with Budgeted Uncertainty}

\author[1]{Marc Goerigk\textsuperscript{\orcidlink{0000-0002-2516-0129}}}
\author[1]{Dorothee Henke\textsuperscript{\orcidlink{0000-0001-9190-642X}}}
\author[2]{Lasse Wulf\textsuperscript{\orcidlink{0000-0001-7139-4092}}}

\affil[1]{Business Decisions and Data Science, University of Passau, Germany. \texttt{\{marc.goerigk,dorothee.henke\}@uni-passau.de}}
\affil[2]{Theoretical Computer Science, IT University of Copenhagen, Denmark. \texttt{lasw@itu.dk}}

\date{}

\maketitle

\begin{abstract}
  A standard type of uncertainty set in robust optimization is budgeted uncertainty, where an interval of possible values for each parameter is given and the total deviation from their lower bounds is bounded. In the two-stage setting, discrete and continuous budgeted uncertainty have to be distinguished. The complexity of such problems is largely unexplored, in particular if the underlying nominal optimization problem is simple, such as for selection problems. In this paper, we give a comprehensive answer to long-standing open complexity questions for three types of selection problems and three types of budgeted uncertainty sets. In particular, we demonstrate that the two-stage selection problem with continuous budgeted uncertainty is NP-hard, while the corresponding two-stage representative selection problem is solvable in polynomial time. Our hardness result implies that also the two-stage assignment problem with continuous budgeted uncertainty is NP-hard.


\end{abstract}

\section{Introduction}

In combinatorial optimization, we often study problems of the type
\begin{equation}
\min_{x\in X} \sum_{i=1}^n c_i x_i, \tag{Nom}\label{eq:nom}
\end{equation}
where $X\subseteq\{0,1\}^n$ denotes the discrete set of feasible solutions, and $c\in\Rge^n$ is a cost vector that represents the costs of each item that we can choose. In practice, problem data is often uncertain, which has led to the popular framework of robust combinatorial optimization, where the uncertainty in the cost vector $c$ is included by a worst-case perspective over a set of possible cost scenarios. The focus of this paper will be on two-stage robustness. Before we discuss this framework, we introduce the more fundamental setting of single-stage robustness.
The robust counterpart of \eqref{eq:nom} then is to solve
\begin{equation}
\min_{x\in X} \max_{c\in U} \sum_{i=1}^n c_i x_i \tag{MinMax} \label{eq:minmax}
\end{equation}
for an uncertainty set $U\subseteq\Rge^n$ that contains all scenarios against which we wish to protect the solution $x$. For a general overview on robust combinatorial optimization, we refer to the recent book \cite{robook}.

Unfortunately, robust problems \eqref{eq:minmax} may become harder to solve than their nominal counterparts \eqref{eq:nom}, even for simple uncertainty sets. Most nominal problems that can be solved in polynomial time become NP-hard in their robust version, even for the case that there are two possible cost scenarios, i.e.,  $U=\{c^1,c^2\}$; see \cite{kasperski2016robust}. As a consequence, there has been intense research on specific special cases of $X$ and $U$ that allow better tractability and establish the boundary cases of complexity.

Regarding the set $X$, variants of so-called selection problems have been studied. In the most general case, we define the nominal \emph{multi-representative selection problem} as follows. Let a set of $n$ items and a partition $T_1\cup T_2 \cup \dots \cup T_m$ of this set be given, along with integers $p_j$ for each $j\in[m]=\{1,\dots,m\}$ such that $p_j\in\{1,\dots,\lvert T_j \rvert\}$. For a cost vector $c\in\Rge^n$, the task is to select $p_j$ items from each set $T_j$ such that the total costs are minimized, i.e., to solve $\min_{x\in X} \sum_{i=1}^n c_ix_i$ with
\[ X = \{ x\in\{0,1\}^n : \sum_{i\in T_j} x_i = p_j \ \forall j\in[m]\}. \]
Two special cases of this problem are the \emph{representative selection problem}, where $p_j=1$ for all $j\in[m]$, and the \emph{selection problem}, where $m=1$. Robust variants of the representative selection problem have been studied, e.g., in \cite{busing2011recoverable} and \cite{dolgui2012min}, while the selection problem has been studied, e.g., in \cite{doerr2013improved} and \cite{lachmann2021linear}. The analysis of types of selection problems in robust optimization plays a central role, as they represent one of the easiest nontrivial combinatorial problems that are possible.

Regarding the set $U$, a seminal breakthrough has been achieved with the introduction of budgeted uncertainty sets in \cite{bertsimas2003robust,bertsimas2004price}, which allow some degree of modeling flexibility, while the complexity class of nominal combinatorial optimization problems is preserved for their robust counterpart. 
Such sets are defined by an interval $[\underline{c}_i,\underline{c}_i+d_i]$ for each item with uncertain costs, and an additional bound~$\Gamma$ on the total cost deviation. 
Given $\underline c_i$ and $d_i$, we define $\overline c_i := \underline c_i + d_i$.
In the current literature, three variants of this concept have been established. The \emph{continuous budgeted uncertainty set} is defined as
\[ U^C = \{ c\in\Rge^n : \exists \delta\in[0,1]^n \text{ s.t. } c_i = \underline{c}_i + d_i\delta_i\ \forall i\in[n],\ \sum_{i=1}^n \delta_i \le \Gamma\}, \]
while the \emph{discrete budgeted uncertainty set} is given by
\[ U^D = \{ c\in\Rge^n : \exists \delta\in\{0,1\}^n \text{ s.t. } c_i = \underline{c}_i + d_i\delta_i\ \forall i\in[n],\ \sum_{i=1}^n \delta_i \le \Gamma\}. \]
Observe that, in the discrete case, item costs are either at their respective lower bound or at their upper bound (in which case we also say that an item is attacked). In the continuous case, it is possible to distribute the attack budget~$\Gamma$ continuously over items. If $\Gamma$ is an integer, then min-max problems with discrete or continuous uncertainty sets result in the same worst-case scenario and thus both problems are equivalent. The third variant of budgeted uncertainty that has been studied in the literature, which we call \emph{alternative continuous budgeted uncertainty}, is defined as
\[ U^{AC} = \{ c\in\Rge^n : \exists \delta\in \prod_{i=1}^n [0,d_i] \text{ s.t. } c_i = \underline{c}_i + \delta_i\ \forall i\in[n],\ \sum_{i=1}^n \delta_i \le \Gamma\}. \]
In this setting, $\delta$ does not represent the (relative) degree by which costs are increased from $\underline{c}$ to $\overline{c}$, but rather the (absolute) amount of cost increase itself. Accordingly, we can assume $\Gamma\le n$ for the first two types of uncertainty, but $\Gamma \le \sum_{i=1}^n d_i$ for the last type of set.
In Figure~\ref{fig:ex}, we visualize the differences between these three sets. Figure~\ref{ex1a} shows a continuous budgeted set with $\underline{c}=(2,1)$, $d = (2,3)$, and $\Gamma=1.5$. The discrete budgeted set with the same parameters is presented in Figure~\ref{ex1b}. Finally, we show the case with $\Gamma=4$ for the alternative continuous budgeted set in Figure~\ref{ex1c}, where the sum of additional costs is bounded rather than the sum of relative deviations.

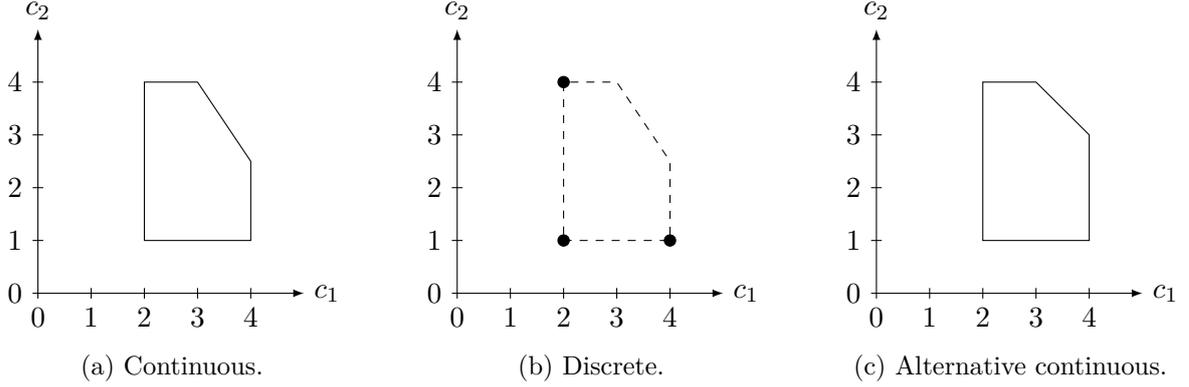
\begin{figure}
  \begin{tabular}{ccc}
    \begin{subfigure}[t]{0.32\textwidth}
      \centering
      \begin{tikzpicture}[scale=0.7]
        \draw[-latex] (-0.1,0) node[left] {$0$} -- (5,0) node[right] {$c_1$};
        \draw[-latex] (0,-0.1) node[below] {$0$} -- (0,5) node[above] {$c_2$};

        \draw (1,-0.1) node[below] {$1$} -- (1,0.1);
        \draw (2,-0.1) node[below] {$2$} -- (2,0.1);
        \draw (3,-0.1) node[below] {$3$} -- (3,0.1);
        \draw (4,-0.1) node[below] {$4$} -- (4,0.1);

        \draw (-0.1,1) node[left] {$1$} -- (0.1,1);
        \draw (-0.1,2) node[left] {$2$} -- (0.1,2);
        \draw (-0.1,3) node[left] {$3$} -- (0.1,3);
        \draw (-0.1,4) node[left] {$4$} -- (0.1,4);

        \draw (2,1) -- (4,1) -- (4,2.5) -- (3,4) -- (2,4) -- cycle;
      \end{tikzpicture}
      \subcaption{Continuous.}
      \label{ex1a}
    \end{subfigure}
    &
    \begin{subfigure}[t]{0.32\textwidth}
      \centering
      \begin{tikzpicture}[scale=0.7]
        \draw[-latex] (-0.1,0) node[left] {$0$} -- (5,0) node[right] {$c_1$};
        \draw[-latex] (0,-0.1) node[below] {$0$} -- (0,5) node[above] {$c_2$};

        \draw (1,-0.1) node[below] {$1$} -- (1,0.1);
        \draw (2,-0.1) node[below] {$2$} -- (2,0.1);
        \draw (3,-0.1) node[below] {$3$} -- (3,0.1);
        \draw (4,-0.1) node[below] {$4$} -- (4,0.1);

        \draw (-0.1,1) node[left] {$1$} -- (0.1,1);
        \draw (-0.1,2) node[left] {$2$} -- (0.1,2);
        \draw (-0.1,3) node[left] {$3$} -- (0.1,3);
        \draw (-0.1,4) node[left] {$4$} -- (0.1,4);

        \draw[dashed] (2,1) -- (4,1) -- (4,2.5) -- (3,4) -- (2,4) -- cycle;
        \node[draw,circle,fill=black,inner sep=1.5pt] at (2,1) {};
        \node[draw,circle,fill=black,inner sep=1.5pt] at (4,1) {};
        \node[draw,circle,fill=black,inner sep=1.5pt] at (2,4) {};
      \end{tikzpicture}
      \subcaption{Discrete.}
      \label{ex1b}
    \end{subfigure}
    &
    \begin{subfigure}[t]{0.32\textwidth}
      \centering
      \begin{tikzpicture}[scale=0.7]
        \draw[-latex] (-0.1,0) node[left] {$0$} -- (5,0) node[right] {$c_1$};
        \draw[-latex] (0,-0.1) node[below] {$0$} -- (0,5) node[above] {$c_2$};

        \draw (1,-0.1) node[below] {$1$} -- (1,0.1);
        \draw (2,-0.1) node[below] {$2$} -- (2,0.1);
        \draw (3,-0.1) node[below] {$3$} -- (3,0.1);
        \draw (4,-0.1) node[below] {$4$} -- (4,0.1);

        \draw (-0.1,1) node[left] {$1$} -- (0.1,1);
        \draw (-0.1,2) node[left] {$2$} -- (0.1,2);
        \draw (-0.1,3) node[left] {$3$} -- (0.1,3);
        \draw (-0.1,4) node[left] {$4$} -- (0.1,4);

        \draw (2,1) -- (4,1) -- (4,3) -- (3,4) -- (2,4) -- cycle;
      \end{tikzpicture}
      \subcaption{Alternative continuous.}
      \label{ex1c}
    \end{subfigure}
  \end{tabular}
  \caption{Example budgeted uncertainty sets.}
  \label{fig:ex}
\end{figure}

Min-max problems as defined in \eqref{eq:minmax} are static in the sense that once a decision is made, it is impossible to react to further developments. In two-stage robust optimization, as first introduced in \cite{ben2004adjustable}, it is possible to decide some variables after the uncertainty has been revealed; see also the survey \cite{yanikouglu2019survey}. In the context of combinatorial optimization, we define the set $Y(x)$ as all second-stage solution vectors that can be obtained for a given first-stage solution~$x$, i.e., $Y(x) = \{ y\in\{0,1\}^n : x+y \in X\}$. The set of feasible first-stage solutions is defined as those vectors $x$ where feasible second-stage solutions exist, i.e., $X' = \{ x\in\{0,1\}^n : Y(x)\neq\emptyset\}$. In the context of multi-representative selection problems, we have
\[ X' = \{x\in\{0,1\}^n : \sum_{i\in T_j} x_i \le p_j \ \forall j\in[m] \} \]
and
\[ Y(x) = \{ y\in\{0,1\}^n : \sum_{i\in T_j} (x_i + y_i) = p_j\ \forall j\in[m],\ x_i + y_i \le 1\ \forall i \in [n] \}. \]
Then, the two-stage robust problem is to solve
\begin{equation}
\min_{x\in X'} \left( \sum_{i=1}^n C_ix_i + \max_{c\in U} \min_{y\in Y(x)} \sum_{i=1}^n c_i y_i \right) \tag{TwoStage}\label{eq:twostage}
\end{equation}
for some given first-stage cost vector $C\in\Rge^n$.

Two-stage selection problems remain some of the few known cases where it is possible to achieve polynomial-time solvability. In  \cite{ChasseinEtAl2018}, it was shown that the two-stage selection problem with alternative continuous budgeted uncertainty can be solved in polynomial time. On the other hand, in \cite{GoerigkLendlWulf2022}, a proof has been given that the same problem with discrete budgeted uncertainty is NP-hard.  In \cite{goerigk2022robust}, hardness was shown for general polyhedral uncertainty sets in outer description, while the case of inner descriptions was discussed in \cite{goerigk2020two}. A polynomial-time algorithm was presented for two-stage representative selection with alternative continuous budgeted uncertainty in \cite{goerigk2022robust}.

To the best of our knowledge, no other positive or negative complexity results for two-stage selection problems with budgeted uncertainty are known. Other decision criteria, such as min-max regret, have been studied in the context of selection problems as well; see \cite{averbakh2004interval} for positive and \cite{averbakh2001complexity} for negative results. Two-stage selection problems with other types of uncertainty sets have been studied in \cite{kasperski2017robust}, and the so-called min-max-min selection problem with continuous budgeted uncertainty was studied in \cite{brauner2024single}. Single-stage robust variants of the selection problem with budgeted uncertainty, which are motivated by the two-stage setting, have been introduced and investigated in \cite{LhommeEtAl2025}.

In this paper, we give a comprehensive answer to all open complexity questions for two-stage robust selection problems with budgeted uncertainty. Table~\ref{table:overview} presents an overview on known and new complexity results for these problems. Observe that hardness results for special cases (selection, representative selection) extend to hardness results for the more general case (multi-representative selection), while polynomial-time solvability extends from the general case towards the special cases.

\begin{table}
  \begin{center}
    \begin{tabular}{rlll}
      \toprule
      & Discrete (D) & Continuous (C) & Alt. Cont. (AC) \\
      \midrule
      Repr. Sel. (2RS) & NP-hard (Thm.~\ref{thm:discrete-uncertainty-hard}) & P (Thm.~\ref{th:reprcontinP}) & P \cite{goerigk2022robust} \\
      Sel. (2S) & NP-hard \cite{GoerigkLendlWulf2022} & NP-hard (Thm.~\ref{thm_hardness})& P  \cite{ChasseinEtAl2018} \\
      Multi-Repr. Sel. (2MRS) & NP-hard \cite{GoerigkLendlWulf2022} & NP-hard (Thm.~\ref{thm_hardness}) & P (Thm.~\ref{th:multialtinP}) \\
      \bottomrule
    \end{tabular}
    \caption{Overview of known and new complexity results for two-stage robust selection problems with budgeted uncertainty.}
    \label{table:overview}
  \end{center}
\end{table}

In particular, we prove that the complexity of the selection and representative selection problems may differ, and so does the complexity between the two variants of continuous budgeted uncertainty sets. These are the first results of this kind, as far as we are aware. Our results also provide a partial answer to the open problem number 7 in \cite{robook}, which is to determine the complexity of two-stage and recoverable robust selection problems with continuous budgeted uncertainty. Our reduction requires a deeper understanding of how the optimal value function behaves depending on the choice of a dual variable, which may explain why these complexity questions have remained open for quite some time (the related recoverable case was first noted as an open problem in \cite{busing2011recoverable}).
Finally, our main hardness result implies that also the two-stage assignment problem with continuous budgeted uncertainty is NP-hard (see Corollary~\ref{cor_hardness_assignment}), which gives a partial answer to the open problem number 15 in \cite{robook}.

Throughout the paper, we use the following shorthands to denote problems we consider. We write ``2RS'' for two-stage representative selection, ``2S'' for two-stage selection, and ``2MRS'' for two-stage multi-representative selection. For the uncertainty sets, we write ``D'' for discrete, ``C'' for continuous, and ``AC'' for alternative continuous budgeted uncertainty. These two problem name components are connected with a dash, e.g., we write (2RS-C) for the two-stage robust representative selection problem with continuous budgeted uncertainty.

The remainder of this paper is structured as follows. In Section~\ref{sec_reformulation}, 
we discuss how to reformulate the two-stage problem \eqref{eq:twostage}, which is central for the further discussion.
In the following sections, we derive the complexity results of Table~\ref{table:overview}. 
We provide polynomial-time algorithms for (2RS-C) in Section~\ref{sec_poly_alg} and discuss the hardness of (2S-C) in Section~\ref{sec:hardnesscont} (which implies hardness of (2MRS-C) as well).
We then show that (2MRS-AC) can be solved in polynomial time in Section~\ref{sec:multi-alt}, before showing hardness of (2RS-D) in Section~\ref{sec:hardnessdisc}. We conclude our paper in Section~\ref{sec:conclusions}.

\section{Reformulation as a Mixed-Integer Linear Program}
\label{sec_reformulation}

In this section, we derive mixed-integer linear programming reformulations of the problems (2MRS-C) and (2MRS-AC), following a standard dualization procedure. The reformulations will be used as the basis for the results in the later sections.

As before, we denote by $n$ the number of items, which are partitioned into
the sets $[n] = T_1 \cup \dots \cup T_m$. Let first-stage costs $C \in \Rge^n$ and an uncertainty set $U\subseteq \Rge^n$ be given. Recall that the two-stage problem we consider is of the form
\[
  \min_{x \in X'} \left( \sum_{i = 1}^n C_i x_i + \max_{c \in U} \min_{y \in Y(x)} \sum_{i = 1}^n c_i y_i\right),
\]
where $X' = \{x \in \{0, 1\}^n : \sum_{i \in T_j} x_i \le p_j ~\forall j \in [m]\}$
and $Y(x) = \{y \in \{0, 1\}^n : \sum_{i \in T_j} (x_i + y_i) = p_j ~\forall j \in [m],\ x_i + y_i \le 1\ \forall i \in [n]\}$ for $x \in X'$.

Observe that
the second-stage problem $\min_{y \in Y(x)} \sum_{i = 1}^n c_i y_i$,
for any fixed $x \in X'$ and $c \in U$,
is equivalent to its linear programming relaxation,
i.e., the integrality of $y$ can be relaxed,
because the matrix in the description of $Y(x)$ is totally unimodular.
The dual of this linear program can be written as follows:
\begin{align*}
  \max~ & \sum_{i = 1}^m \alpha_j (p_j - \sum_{i \in T_j} x_i) + \sum_{i = 1}^n \beta_i (1 - x_i) \\
  \st~ & \alpha_j + \beta_i \le c_i & \forall j \in [m], i \in T_j \\
        & \alpha \in \R^m, \beta \in \Rle^n
\end{align*}

For ease of presentation, let us begin with continuous budgeted uncertainty sets.
The adversarial problem $\max_{c \in U} \min_{y \in Y(x)} \sum_{i = 1}^n c_i y_i$,
for any fixed $x \in X'$,
is then equivalent to the following linear program:
\begin{align*}
  \max~ & \sum_{i = 1}^m \alpha_j (p_j - \sum_{i \in T_j} x_i) + \sum_{i = 1}^n \beta_i (1 - x_i) \\
  \st~ & \alpha_j + \beta_i \le \underline{c}_i + \delta_i d_i & \forall j \in [m], i \in T_j \\
        & \sum_{i = 1}^n \delta_i \le \Gamma \\
        & \alpha \in \R^m, \beta \in \Rle^n, \delta \in [0, 1]^n
\end{align*}
The dual of this linear program can be written as follows:
\begin{align*}
  \min~ & \sum_{i = 1}^n \underline{c}_i y_i + \Gamma \pi + \sum_{i = 1}^n \rho_i \\
  \st~ & \sum_{i \in T_j} (x_i + y_i) = p_j & \forall j \in [m] \\
        & x_i + y_i \le 1 & \forall i \in [n] \\
        & \pi + \rho_i \ge d_i y_i & \forall i \in [n] \\
        & y \in [0, 1]^n, \pi \in \Rge, \rho \in \Rge^n
\end{align*}
Combining this formulation with the first-stage decision
finally results in the following mixed-integer programming formulation
of (2MRS-C):
\begin{align*} \label{2MRS-C} \tag{2MRS-C}
  \min~ & \sum_{i = 1}^n C_i x_i + \sum_{i = 1}^n \underline{c}_i y_i + \Gamma \pi + \sum_{i = 1}^n \rho_i \\
  \st~ & \sum_{i \in T_j} (x_i + y_i) = p_j & \forall j \in [m] \\
        & x_i + y_i \le 1 & \forall i \in [n] \\
        & \pi + \rho_i \ge d_i y_i & \forall i \in [n] \\
        & x \in \{0, 1\}^n, y \in [0, 1]^n, \pi \in \Rge, \rho \in \Rge^n
\end{align*}
For (2MRS-AC), the same dualization procedure gives the following problem formulation:
\begin{align*} \label{2MRS-AC} \tag{2MRS-AC}
  \min~ & \sum_{i = 1}^n C_i x_i + \sum_{i = 1}^n \underline{c}_i y_i + \Gamma \pi + \sum_{i = 1}^n d_i \rho_i \\
  \st~ & \sum_{i \in T_j} (x_i + y_i) = p_j & \forall j \in [m] \\
        & x_i + y_i \le 1 & \forall i \in [n] \\
        & \pi + \rho_i \ge y_i & \forall i \in [n] \\
        & x \in \{0, 1\}^n, y \in [0, 1]^n, \pi \in \Rge, \rho \in \Rge^n
\end{align*}
Note that the position of the $d_i$ coefficients has changed. For discrete budgeted uncertainty sets, compact reformulations are possible (see \cite{ChasseinEtAl2018}), but not necessary for this paper.

\section{Representative Selection with Continuous Budgeted Sets}
\label{sec_poly_alg}

In this section, we consider problem (2RS-C), i.e., problem~\eqref{2MRS-C} with $p_j=1$ for all $j \in [m]$.
Note that the constraints $x_i + y_i \le 1$ in \eqref{2MRS-C} are redundant in this case.
We show that (2RS-C)
can be solved in time~$\bigO(n \log n)$
by a combinatorial algorithm.
As input, we are given cost coefficients $C_i, \underline c_i, d_i \in \Rge$ for $i \in [n]$, together with a budget parameter~$\Gamma \in \Rge$ and a partition $[n] = T_1 \cup \dots \cup T_m$.
In order to reduce the technical details necessary to present our result, throughout this section, we make the simplifying assumption that $d_i \neq 0$ for all $i \in [n]$.
Hence, division by $d_i$ is always defined. We explain in the end of the section how this assumption can be removed.

Observe that,
if $\pi \in \Rge$ is fixed,
the problem decomposes into $m$ independent problems,
one for each $j \in [m]$:
\begin{align*}
  \min~ & \sum_{i \in T_j} C_i x_i + \sum_{i \in T_j} \underline{c}_i y_i + \sum_{i \in T_j} \rho_i \\
  \st~ & \sum_{i \in T_j} (x_i + y_i) = 1 \\
        & \pi + \rho_i \ge d_i y_i & \forall i \in T_j \\
        & x \in \{0, 1\}^{T_j}, y \in [0, 1]^{T_j}, \rho \in \Rge^{T_j}
\end{align*}

For an optimal solution to this problem,
there are the following two options
(whichever is cheaper):
\begin{itemize}
\item In the first case,
one item~$i \in T_j$ is selected in the first stage,
i.e., $x_i = 1$ and all other $x$ and $y$ variables are set to $0$.
Then $i$ must be an item with minimal first-stage cost~$C_i = \min_{k \in T_j} C_k$.
Moreover,
$\rho$ only has to satisfy the constraints $\rho_i \ge - \pi$ for all $i \in T_j$
and can therefore be set to $0$.
\item In the second case,
no item is selected in the first stage,
i.e., $x_i = 0$ for all $i \in T_j$.
We then get an optimal solution of the following problem \eqref{fjpi}.
Denote its optimal value,
depending on $\pi$,
by the function~$f_j \colon \Rge \to \Rge$.
\begin{align*} \tag{\textasteriskcentered} \label{fjpi}
  f_j(\pi) := \min~ & \sum_{i \in T_j} \underline{c}_i y_i + \sum_{i \in T_j} \rho_i \\
  \st~ & \sum_{i \in T_j} y_i = 1 \\
        & \pi + \rho_i \ge d_i y_i & \forall i \in T_j \\
        & y \in [0, 1]^{T_j}, \rho \in \Rge^{T_j}
\end{align*}
\end{itemize}

By the above observations,
the minimal achievable value of (2RS-C),
for any fixed $\pi$,
can be described by the function $f \colon \Rge \to \Rge$ with
\begin{align*} \tag{\textasteriskcentered \textasteriskcentered} \label{fpi}
  f(\pi) := \Gamma \pi + \sum_{j = 1}^m \min \{ \min_{k \in T_j} C_k, f_j(\pi) \}.
\end{align*}
The optimal value of (2RS-C)
is thus given by $\min_{\pi \in \Rge} f(\pi)$.

In the following,
we will study the functions $f_j$
and show that they are piecewise linear and continuous.
Hence,
also the function $f$ is piecewise linear and continuous,
and its minimum can be found by iterating over its breakpoints.
By describing how to compute the sets of the breakpoints of the functions $f_j$
and then of the function $f$,
we show how this can be done in polynomial time.

We first reformulate \eqref{fjpi} into the following equivalent problem:
\begin{align*} \tag{\textasteriskcentered'} \label{fjpi'}
  \min~ & \sum_{i \in T_j} \underline{c}_i \underline{y}_i + \sum_{i \in T_j} \overline{c}_i \overline{y}_i \\
  \st~ & \sum_{i \in T_j} (\underline{y}_i + \overline{y}_i) = 1 \\
        & \underline{y}_i \in [0, \pi_i] & \forall i \in T_j \\
        & \overline{y}_i \in [0, 1 - \pi_i] & \forall i \in T_j,
\end{align*}
where $\pi_i := \min\{1, \pi/d_i\}$ for all $i \in T_j$.
Recall that $\overline{c} = \underline{c} + d$.

\begin{lemma} \label{lem_fjpi_fjpi'}
  The problems \eqref{fjpi} and \eqref{fjpi'} are equivalent.
\end{lemma}

\begin{proof}
First, let $y, \rho$ be an optimal solution to \eqref{fjpi}.
Define $\underline{y}_i$ and $\overline{y}_i$ as follows, for all $i \in T_j$:
\begin{itemize}
\item If $y_i \le \pi_i$, then set $\underline{y}_i := y_i$ and $\overline{y}_i := 0$.
\item If $y_i > \pi_i$, then set $\underline{y}_i := \pi_i$ and $\overline{y}_i := y_i - \pi_i$.
\end{itemize}
Note that this always ensures $y = \underline{y} + \overline{y}$.
Therefore, $\underline{y}$ and $\overline{y}$ form a feasible solution to \eqref{fjpi'}.

The above two cases are closely related to
the constraints $\pi + \rho_i \ge d_i y_i$ in \eqref{fjpi}:
\begin{itemize}
\item If $y_i \le \pi_i$, then $y_i \le \pi/d_i$ and it is optimal to set $\rho_i = 0$.
\item If $y_i > \pi_i$, which is only possible if $\pi_i = \pi/d_i < 1$, then it is optimal to set $\rho_i = d_i y_i - \pi = d_i (y_i - \pi_i)$.
\end{itemize}
Hence, we may assume that $\rho$ from the given optimal solution has these properties.
This implies that the objective value of the solution $\underline{y}, \overline{y}$ in \eqref{fjpi'} is
\begin{align*}
  \sum_{i \in T_j} \underline{c}_i \underline{y}_i + \sum_{i \in T_j} \overline{c}_i \overline{y}_i
  &= \sum_{i \in T_j} \underline{c}_i y_i + \sum_{i \in T_j} d_i \overline{y}_i \\
  &= \sum_{i \in T_j} \underline{c}_i y_i + \sum_{i \in T_j, y_i > \pi_i} d_i (y_i - \pi_i) \\
  &= \sum_{i \in T_j} \underline{c}_i y_i + \sum_{i \in T_j} \rho_i,
\end{align*}
i.e., the same value as the objective value of $y, \rho$ in \eqref{fjpi}.

For the other direction,
let $\underline{y}$ and $\overline{y}$ be an optimal solution to \eqref{fjpi'}.
Set $y: = \underline{y} + \overline{y}$,
and define $\rho_i$ as follows, for all $i \in T_j$:
\begin{itemize}
\item If $\underline{y}_i + \overline{y}_i \le \pi_i$, then set $\rho_i: = 0$.
\item If $\underline{y}_i + \overline{y}_i > \pi_i$, then set $\rho_i: = d_i y_i - \pi$.
\end{itemize}
This is a feasible solution to \eqref{fjpi}.

Observe that,
as $\underline{y}$ and $\overline{y}$ form an optimal solution of \eqref{fjpi'}
and $\underline{c}_i \le \overline{c}_i$,
we may assume that the following holds for all $i \in T_j$:
\begin{itemize}
\item If $\underline{y}_i + \overline{y}_i \le \pi_i$, then $\overline{y}_i = 0$.
\item If $\underline{y}_i + \overline{y}_i > \pi_i$, then $\underline{y}_i = \pi_i$ and $\pi_i = \pi/d_i < 1$.
\end{itemize}
Therefore,
the objective value of the solution $y, \rho$ in \eqref{fjpi}
can be written as
\begin{align*}
  \sum_{i \in T_j} \underline{c}_i y_i + \sum_{i \in T_j} \rho_i
  &= \sum_{i \in T_j} \underline{c}_i (\underline{y}_i + \overline{y}_i) + \sum_{i \in T_j, \underline{y}_i + \overline{y}_i > \pi_i} (d_i (\underline{y}_i + \overline{y}_i) - \pi) \\
  &= \sum_{i \in T_j} \underline{c}_i \underline{y}_i + \sum_{i \in T_j} \underline{c}_i \overline{y}_i + \sum_{i \in T_j, \underline{y}_i + \overline{y}_i > \pi_i} (d_i \underline{y}_i - \pi) + \sum_{i \in T_j} d_i \overline{y}_i \\
  &= \sum_{i \in T_j} \underline{c}_i \underline{y}_i + \sum_{i \in T_j} \overline{c}_i \overline{y}_i.
\end{align*}

Thus, both problems \eqref{fjpi} and \eqref{fjpi'} have the same optimal value,
and any optimal solution to~\eqref{fjpi} can easily be transformed into an optimal solution to \eqref{fjpi'},
and vice versa.
\end{proof}

Scaling the ranges of all variables in \eqref{fjpi'} to $[0, 1]$
leads to another equivalent problem formulation,
which now has the structure of a continuous knapsack problem (in the minimization version):
\begin{align*} \tag{\textasteriskcentered''} \label{fjpi''}
  \min~ & \sum_{i \in T_j} \underline{c}_i \pi_i \underline{z}_i + \sum_{i \in T_j} \overline{c}_i (1 - \pi_i) \overline{z}_i \\
  \st~ & \sum_{i \in T_j} (\pi_i \underline{z}_i + (1 - \pi_i) \overline{z}_i) = 1 \\
        & \underline{z}_i, \overline{z}_i \in [0, 1] & \forall i \in T_j
\end{align*}

\begin{lemma} \label{lem_fjpi'_fjpi''}
  The problems \eqref{fjpi'} and \eqref{fjpi''} are equivalent.
\end{lemma}

\begin{proof}
  First observe that both problems are equivalent for $\pi=0$, so let us assume $\pi>0$.
  Given an optimal solution $\underline{y}, \overline{y}$ to \eqref{fjpi'},
  setting $\underline{z}_i := \underline{y}_i / \pi_i$ and $\overline{z}_i := \overline{y}_i / (1 - \pi_i)$ (or $\overline{z}_i := 0$ in case $\pi_i = 1$)
  leads to a feasible solution to \eqref{fjpi''} of the same objective value.
  
  Analogously,
  given an optimal solution $\underline{z}, \overline{z}$ to \eqref{fjpi''},
  setting $\underline{y}_i := \pi_i \underline{z}_i$ and $\overline{y}_i := (1 - \pi_i) \overline{z}_i$
  leads to a feasible solution to \eqref{fjpi'} of the same objective value.
\end{proof}

The structure of the problem \eqref{fjpi''} is now simple enough that it can be solved in a greedy fashion,
analogously to Dantzig's algorithm for the continuous knapsack problem~\cite{Dantzig1957}:
Sort the $2 \lvert T_j \rvert$ items corresponding to the variables $\underline{z}_i$ and $\overline{z}_i$
by their ratios of cost (i.e., the variable's coefficient in the objective function) and size (i.e., the variable's coefficient in the constraint),
in a nondecreasing order.
Formally, consider the set $Z := \bigcup_{i \in T_j}\set{\underline z_i, \overline z_i}$.
For $i \in T_j$, let $\ratio(\underline z_i) :=   \underline{c}_i$ (which is equal to $\underline{c}_i \pi_i / \pi_i$ for $\pi_i \neq 0$) 
and $\ratio(\overline z_i) := \overline{c}_i$ (which is equal to $\overline{c}_i (1 - \pi_i) / (1 - \pi_i)$ for $\pi_i \neq 1$).
The algorithm computes an order $O_Z$ of $Z$ that sorts the items in $Z$ nondecreasingly by their ratio. Then it selects items in this order until the total size of $1$ in the constraint is reached, possibly selecting the last item only fractionally.

\begin{figure}
  \centering
  \includegraphics[scale=1.05]{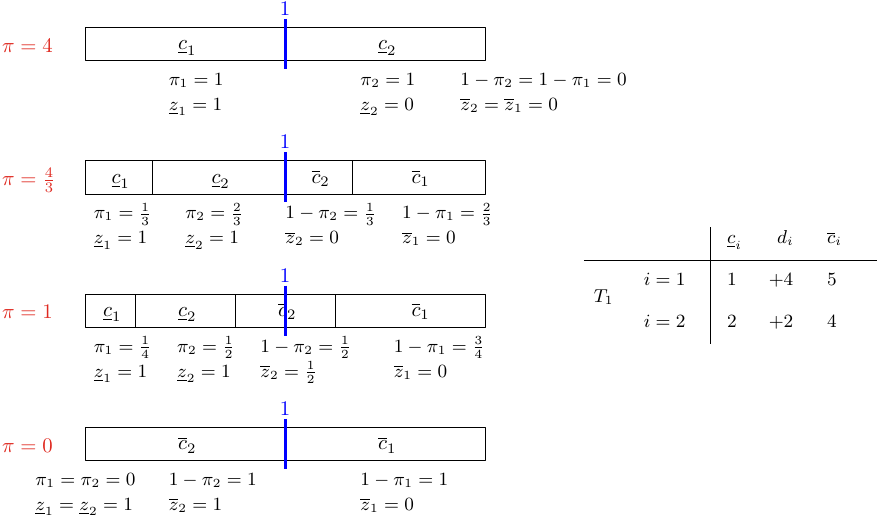}
  \caption{Problem \eqref{fjpi''} can be understood as a continuous min-knapsack problem dependent on parameter $\pi$.
    We want to pack 1 unit size into the knapsack while minimizing the cost. Index $i$ is associated with two items of size $\pi_i = \min\{1,\pi/d_i\}$ and $1 - \pi_i$, respectively.
    We depict an example instance with $m=1$ and $\lvert T_1 \rvert=2$ for different values of $\pi$.
    The order of the four items is given by their ratios $\underline{c}_1 = 1 < \underline{c}_2 = 2 < \overline{c}_2 = 4 < \overline{c}_1 = 5$.}
  \label{fig:parametric-knapsack}
\end{figure}

Observe that $\ratio(z)$ is actually independent of $\pi$ for all $z \in Z$.
Hence, when the problem~\eqref{fjpi''} is solved for different values of $\pi$,
we may assume that the order $O_Z$ of the items does not change.
Only the second part of the algorithm behaves differently,
as the items' sizes change and therefore the total size of $1$ might be reached earlier or later when iterating through the ordered items.
Figure~\ref{fig:parametric-knapsack} showcases an example.

Based on these insights,
we will now investigate in more detail
how the optimal value of \eqref{fjpi''} changes when $\pi$ changes.
Because of Lemmas \ref{lem_fjpi_fjpi'} and \ref{lem_fjpi'_fjpi''},
this tells us how the function $f_j$ behaves.

For ease of notation, we assume $T_j = [n_j]$ for $n_j := \lvert T_j \rvert$ in the following.
Since $\underline{c}_i \le \overline{c}_i$ for all $i \in T_j$,
we may assume,
without loss of generality,
that the order of the ratios corresponding to~$O_Z$ starts with $\underline{c}_1 \le \underline{c}_2 \le \dots \le \underline{c}_r \le \overline{c}_s$,
for some $r \in T_j = [n_j]$ and some $s \in [r]$.
It will turn out that the algorithm never packs anything that appears later in the order, for any $\pi$.

We start by considering large values of $\pi$ and, step by step, study the algorithm's behavior when decreasing $\pi$.
(Figure~\ref{fig:parametric-knapsack} shows an example of this evolution and Figure~\ref{fig:function-fj-a} depicts the corresponding function $f_j(\pi)$ for this example.)
If $\pi \ge d_1$,
then $\pi_1 = 1$ and the algorithm only packs the first item, i.e., sets $\underline{z}_1 = 1$ and all other $z$ variables to $0$.
The cost of this solution is $\underline{c}_1$.
Hence, for $\pi \ge d_1$, the function $f_j$ is constant with $f_j(\pi) = \underline{c}_1$.

When decreasing $\pi$, i.e., when setting $\pi = d_1 - \varepsilon$ for some (small) $\varepsilon > 0$,
then $\pi_1 = \pi/d_1$ becomes smaller than $1$.
Therefore, the algorithm starts to pack the second item in order to reach the required total size of $1$.
Since both the first and the second item become smaller when decreasing $\pi$,
solutions consisting of the first item and some (increasing) fraction of the second item are selected
until $\pi$ is so small that the two items have a total size of 1.
This happens when $\pi_1 + \pi_2 = \pi/d_1 + \pi/d_2 = 1$, i.e., at $\pi = 1/(1/d_1 + 1/d_2)$.

For $\pi \in [1/(1/d_1 + 1/d_2), d_1]$,
we can describe the solutions and their costs more precisely as follows:
The fraction of the second item that is needed to fill up to a total size of 1
is $\underline{z}_2 = (1 - \pi_1)/\pi_2 = (1 - \pi/d_1)/\pi_2$.
The total cost of the solution is therefore $f_j(\pi) = \underline{c}_1 \pi_1 + \underline{c}_2 \pi_2 \underline{z}_2 = \underline{c}_1 \pi/d_1 + \underline{c}_2 (1 - \pi/d_1) = \underline{c}_2 + \pi (\underline{c}_1 - \underline{c}_2)/d_1$.
This shows that the function $f_j$ is linear with a slope of $(\underline{c}_1 - \underline{c}_2)/d_1 \le 0$
in the range $\pi \in [1/(1/d_1 + 1/d_2), d_1]$.
Note that the fraction $\underline{z}_2$ that is selected of the second item changes nonlinearly,
while the size and the cost of the fractional item change linearly.
When decreasing $\pi$ further, starting from $\pi = 1/(1/d_1 + 1/d_2)$,
the third item starts to be packed in a similar fashion.

Generalizing the observations from the previous paragraphs,
we obtain the following Lemma~\ref{lem_fj}.
Given $j \in [m]$, let us define the order $O_Z = O_Z^{(j)}$ as above.
Set $n_j := \lvert T_j \rvert$ and let $d_1^{(j)}, \dots, d^{(j)}_{n_j}$ be a reordering of the coefficients $(d_i)_{i \in T_j}$ according to the appearance of the corresponding variables~$(\underline{z}_i)_{i \in T_j}$ in the order~$O_Z^{(j)}$.
Moreover, define the indices~$r^{(j)}, s^{(j)} \in [n_j]$ as above, corresponding to the beginning of the order~$O_Z^{(j)}$.
Now define the values
\[
  b_l^{(j)} := \left(\sum_{l' = 1}^l \frac{1}{d^{(j)}_{l'}} \right)^{-1} \text{ for all } l \in [r^{(j)}], \quad b_{r^{(j)} + 1}^{(j)} := 0,
\]
and the set 
\[
  B_j := \left\{b_l^{(j)} : l \in [r^{(j)} + 1]\right\}.
\]
Then $B_j$ has size $\bigO(n_j)$ and we can show the following result.

\begin{lemma} \label{lem_fj}
  For all $j \in [m]$, the function $f_j$ is continuous, convex, and piecewise linear with breakpoints in the set $B_j$.
  The function values in the breakpoints are given by
  \[
    f_j(b_l^{(j)}) = b_l^{(j)} \sum_{l' = 1}^l \frac{\underline{c}_{l'}}{d_{l'}^{(j)}} \text{ for all } l \in [r^{(j)}], \quad f_j(b_{r^{(j)} + 1}^{(j)}) = \overline{c}_{s^{(j)}}.
  \]
\end{lemma}

\begin{proof}
  First, by a general argument, each function $f_j$ is continuous, convex, and piecewise linear. Indeed, by definition, its values equal the optimal values of the linear program~\eqref{fjpi} for varying values of $\pi$. For each fixed $\pi$, \eqref{fjpi} is bounded, and $\pi$ appears in its right-hand side vector. More precisely, the right-hand side vector is an affine function with respect to $\pi$. Therefore, the linear program's optimal value, as a function of $\pi$, is piecewise linear and convex (see, e.g., \cite[Section~5.2]{bertsimas1997}).

  The rest of this proof will be concerned with investigating the piecewise linear structure in more detail, in order to describe the breakpoints.
  Fix some $j \in [m]$, and for ease of notation, assume $T_j = [n_j]$ and let $r := r^{(j)}$, $s := s^{(j)}$, and $d_{l'} := d^{(j)}_{l'}$ for $l' \in [n_j]$, i.e., the order $O_Z^{(j)}$ starts with $\underline{c}_1 \le \underline{c}_2 \le \dots \le \underline{c}_r \le \overline{c}_s$.
  Moreover, write $b_l := b_l^{(j)}$ for $l \in [r + 1]$.
  Note that the values~$b_l$, as defined above, are monotonically decreasing in $l$, i.e., the indices give an ordering of the dedicated breakpoints from right to left.
  Denote the actual breakpoints of the piecewise linear function $f_j$ (in the most general case) by $\tilde{b}_1, \tilde{b}_2, \dots$, also from right to left.
  We will argue that these breakpoints $\tilde{b}_l$ agree with the values $b_l$.

  Recall that $\pi_{l'} = \min\{1, \pi/d_{l'}\}$ for all $l' \in [n]$. For $\pi \geq d_1$, we have $\pi_1 = 1$. Therefore, in this range, only the first item is packed and the function $f_j(\pi)$ is constant equal to $\underline c_1$.
  This gives rise to the rightmost breakpoint $\tilde{b}_1 = d_1 = b_1$ with $f_j(b_1) = \underline{c}_1$.
  
  For $l \in \{2, \dots, r\}$, a solution consisting of the first $l - 1$ items and a fraction of the $l$-th item is selected in the range $\pi \in [\tilde{b}_l, \tilde{b}_{l-1}]$. The breakpoint~$\tilde{b}_l$ corresponds to the solution where the first $l$~items are packed. It is therefore characterized by the equation $\sum_{l'=1}^l \pi_{l'} = \sum_{l'=1}^l \pi/d_{l'} = 1$. This equation is clearly solved by $\pi = 1/(\sum_{l' = 1}^l 1/d_{l'}) = b_l$. Hence, $\tilde{b}_l = b_l$ for all $l \in [r]$.
  The cost of the solution where the first $l \in [r]$~items are packed, i.e., the value of $f_j$ at $\pi = b_l$, is given by
  \begin{align*}
    f_j(b_l) = \sum_{l' = 1}^{l} \underline{c}_{l'} \pi_{l'} = \sum_{l' = 1}^l \underline{c}_{l'} \frac{\pi}{d_{l'}} = \sum_{l' = 1}^l \underline{c}_{l'} \frac{b_l}{d_{l'}} = b_l \sum_{l' = 1}^l \frac{\underline{c}_{l'}}{d_{l'}}.
  \end{align*}

  Finally, when decreasing $\pi$ below $b_r$, we start packing the item $\overline{z}_s$.
  Note that no item after $\overline{z}_s$ is packed:
  The total size of all items in the order up to $\overline{z}_s$ is always at least $1$
  because the items $\underline{z}_s$ and $\overline{z}_s$ have a total size of $\pi_s + (1 - \pi_s) = 1$ already, for any $\pi$.
  When $\pi$ reaches $0$, the size $1 - \pi_s$ of the item $\overline{z}_s$ becomes $1$
  and all previous items vanish (they have size and cost $0$ now).
  This leads to the left end of the function $f_j$:
  The leftmost linear piece,
  corresponding to the item $\overline{z}_s$,
  has the range $\pi \in [0, b_r]$.
  Hence, the leftmost breakpoint is given by $\tilde{b}_{r+1} = 0 = b_{r + 1}$ and $f_j(b_{r + 1}) = \overline{c}_s$.
\end{proof}

Recall that the functions $f_j$ are a part of the description of the function $f$;
see \eqref{fpi}.
We can now derive that also the function $f$ is continuous and piecewise linear:

For every $j \in [m]$, the term $\min_{k \in T_j} C_k$ is a constant,
therefore $\min\{\min_{k \in T_j} C_k, f_j(\pi)\}$ arises from $f_j(\pi)$
by replacing its left part that is larger than this constant (if it exists) by a constant function.
Note that the resulting function is not convex anymore and a new breakpoint $b^\star_j$ might arise at the intersection of $f_j$ and the constant function
 (see Figure~\ref{fig:function-fj-b}).

In the next step, several functions of this type, one for each $j \in [m]$, are (pointwise) added.
This leads to a continuous piecewise linear function again.
Each breakpoint of one of the functions can lead to a breakpoint in their sum.

Finally,
the linear function $\Gamma \pi$ is added,
which also preserves continuity and the piecewise linear structure
(see Figure~\ref{fig:function-fj-c}).
Thus, we have:

\begin{lemma} \label{lem_f}
  The function $f$ is continuous and piecewise linear. Its $\bigO(n)$ breakpoints are contained in the set $\bigcup_{j \in [m]} (B_j \cup \set{b^\star_j})$.
\end{lemma}

In order to solve the problem (2RS-C),
we need to determine the minimum of the function $f$ over all $\pi \ge 0$.
This can be done by iterating through all its breakpoints.

\begin{figure}
  \centering
  \begin{tabular}{ccc}
  \begin{subfigure}[t]{0.3\textwidth}
    \includegraphics[width=\textwidth, page=1]{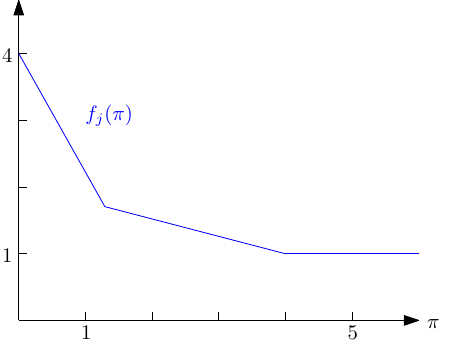}
    \caption{Function $f_j$ is piecewise linear and convex.}
    \label{fig:function-fj-a}
  \end{subfigure}
  &
  \begin{subfigure}[t]{0.3\textwidth}
    \includegraphics[width=\textwidth, page=2]{function-fj}
    \caption{Taking the pointwise minimum of $f_j$ and a constant introduces at most one additional breakpoint.}
    \label{fig:function-fj-b}
  \end{subfigure}
  &
  \begin{subfigure}[t]{0.3\textwidth}
    \includegraphics[width=\textwidth, page=3]{function-fj}
    \caption{Function $f$ is piecewise linear as a sum of the piecewise linear functions from (b) for all $j \in [m]$ and the linear function~$\Gamma \pi$.}
    \label{fig:function-fj-c}
  \end{subfigure}
  \end{tabular}
\caption{Illustration of the piecewise linear functions $f_j$ and $f$.}
\label{fig:function-fj}
\end{figure}

\begin{theorem}\label{th:reprcontinP}
  The problem (2RS-C) can be solved in $\bigO(n \log n)$ time.
\end{theorem}

\begin{proof}
  For each $j \in [m]$, we can use a sorting algorithm to compute the order $O_Z^{(j)}$, and therefore compute the set~$B_j$ of breakpoints of $f_j$ in $\bigO(n_j \log n_j)$ time.
  Note that, since $B_j$ is defined in terms of partial sums, we can compute the whole set $B_j$ in only $\bigO(n_j)$ time once the order is given.
  For each breakpoint $b \in B_j$, we compute and store also the corresponding value~$f_j(b)$.
  By Lemma~\ref{lem_fj}, also these values can be expressed using partial sums. Hence, they can be computed in a total time of $\bigO(n_j)$ as well.
  
  By iterating over the breakpoints $B_j$ in $\bigO(n_j)$ time, we can find the linear piece that intersects with the constant function $\min_{k \in T_j} C_k$ and compute the additional breakpoint $b^\star_j$ (or decide that $b^\star_j$ does not exist).
  This defines the function $g_j(\pi) := \min\{\min_{k \in T_j} C_k, f_j(\pi)\}$ with breakpoints $B_j' := \{0, b^\star_j\} \cup \{b \in B_j : b > b^\star_j\}$.
  We then obtain the set $B := \bigcup_{j \in [m]} B_j'$ of breakpoints of $f$ and sort it in $\bigO(n \log n)$ time.
  
  Finally, we need to evaluate $f(\pi)$ for all $\pi \in B$ in order to find the minimum $\pi^\star$.
  For this, we first compute the value $f(0) = \sum_{j=1}^m g_j(0)$ and then iterate over all $\bigO(n)$ breakpoints in~$B$ in increasing order, determining the value of $f$ at each of them. In order to perform each evaluation in constant time, we keep track of the current slope of $f$. The slope of a linear piece of~$f$ is the sum of $\Gamma$ and the sum of the slopes of the linear pieces of all $g_j$ in the corresponding range. Each slope of a linear piece of $g_j$ can be obtained from the two adjacent breakpoints and their function values in $\bigO(1)$ time. In each iteration corresponding to a breakpoint in $B$, we adjust the current slope of $f$ according to the slope change of the function $g_j$ whose breakpoint is currently considered.
  This enables a total running time of $\bigO(n)$ for the whole loop.

  In order to compute the corresponding first-stage selection $x \in \{0, 1\}^n$ that attains the same minimum for (2RS-C),
  we need to keep track of the set $\{j \in [m] : b^\star_j \ge \pi\}$ during the search for~$\pi^\star$.
  This set for $\pi = \pi^\star$ describes for which $T_j$ the first-stage solution $x$ should select a cheapest item (with costs $\min_{k \in T_j} C_k$), while all other $x$ variables should be set to $0$.
  Hence, $x$~can be determined in $\bigO(n)$ in the end of the algorithm.
\end{proof}

\textbf{Discussion of the case $d_i = 0$.} Finally, we talk about the case where our simplifying assumption does not hold and we have $d_i = 0$ for at least one $i \in [n]$. 
In this case, $\pi/d_i$ is undefined. However, it can be shown that all our arguments are still valid. To this end, we define
\[
\pi_i := \min\set{1, \pi/d_i} \text{ if $d_i \neq 0$, and }\pi_i := 1 \text{ if $d_i = 0$}.
\]
Then it can be shown that Lemma~\ref{lem_fjpi_fjpi'} is still true. Informally, this is due to the fact that $\pi \geq 0$ and so $\pi/d_i$ can be interpreted as positive infinity.
Lemma~\ref{lem_fjpi'_fjpi''} remains true since it describes only a linear scaling.
Lemma~\ref{lem_fj} remains true for the following reason. 
Consider $j \in [m]$ such that for some $i \in T_j$ we have $d_i = 0$. 
Observe that $\ratio(\underline z_i) = \underline c_i =\overline c_i = \ratio(\overline z_i)$ are well-defined, even if $d_i = 0$. Hence, the order $O_Z^{(j)}$ is also well-defined and 
we may assume that $\underline z_i, \overline z_i$ are consecutive items in~$O_Z^{(j)}$.
Analogously to before, without loss of generality, assume the indexing of the items in $T_j$ to be such that $\underline c_1 \leq \underline c_2 \leq \dots \leq \underline c_r \leq \overline c_{s}$ where $\overline c_{s}$ is the first occurrence of some item $\overline z_{s}$ in the order ($s \leq r$).
Then we have $r \leq i$ since the items $\underline z_i, \overline z_i$ are consecutive. We can then consider the definition of the breakpoints $b_l^{(j)}$. For all $l < r$, there is no problem since we do not divide by 0. 
For the last case of $l = r$, we can encounter a division by 0. Then it can be shown that this value can be treated as a breakpoint at ``$0 = \infty^{-1}$'', which is already included in the set $B_j$ (as $b_{r + 1} = 0$).
Lemma~\ref{lem_f} and Theorem~\ref{th:reprcontinP} do not depend on $d_i \neq 0$. In total, we have shown how to solve (2RS-C) in $\bigO(n \log n)$ time, even if $d_i = 0$ for some $i \in [n]$.

\section{Selection with Continuous Budgeted Sets}
\label{sec:hardnesscont}

In this section, we consider the problem (2S-C),
i.e., problem~\eqref{2MRS-C} with $m = 1$,
and prove its NP-hardness.
Let us first recall the problem setting. Given are a set of $n$ items, denoted by $[n] = \{1, \dots, n\}$,
a number $p \in [n]$,
first-stage costs $C \in \Rge^n$,
and an uncertainty set~$U \subseteq \Rge^n$ of possible second-stage costs,
given as
\[
  U = \{c \in \Rge^n : \exists \delta \in [0, 1]^n \text{ s.t. } c_i = \underline{c}_i + d_i \delta_i ~\forall i \in [n], \sum_{i = 1}^n \delta_i \le \Gamma\}
\]
with $\underline{c}, d \in \Rge^n$ and $\Gamma \in \Rge$.
As before, let $\overline{c} = \underline{c} + d$.
The task is to select $p$ items from $[n]$, either in the first stage, where item $i$ has a cost of $C_i$,
or in the second stage,
where item $i$ has a cost of $c_i$,
after the (worst-case) second-stage costs $c \in U$ have been realized.
As a special case of the formulation~\eqref{2MRS-C} derived in Section~\ref{sec_reformulation},
the problem can be written as:
\begin{align*} \label{2S-C} \tag{2S-C}
  \min~ & \sum_{i = 1}^n C_i x_i + \sum_{i = 1}^n \underline{c}_i y_i + \Gamma \pi + \sum_{i = 1}^n \rho_i \\
  \st~ & \sum_{i = 1}^n (x_i + y_i) = p \\
        & x_i + y_i \le 1 & \forall i \in [n] \\
        & \pi + \rho_i \ge d_i y_i & \forall i \in [n] \\
        & x \in \{0, 1\}^n, y \in [0, 1]^n, \pi \in \Rge, \rho \in \Rge^n
\end{align*}

In order to understand this problem, we will study it for fixed $x$ and fixed $\pi$.
For all first-stage solutions $x \in \{0, 1\}^n$ that are feasible for \eqref{2S-C}, i.e., with $\sum_{i = 1}^n x_i \le p$, we define the set $R_x := \{i \in [n] : x_i = 0\}$ of items not selected in the first stage and the remaining budget $p_x := p - \sum_{i = 1}^n x_i$.
Observe that, for an item $i \in [n]$ with $x_i = 1$, the constraint $x_i + y_i \le 1$ forces the corresponding variable $y_i$ to be $0$, and the corresponding variable $\rho_i$ is set to $0$ in any optimal solution as well. Therefore, when $x$ is fixed, we may omit these variables and focus on the indices in~$R_x$.
We now introduce a function $g$ that describes the optimal values of the remaining problem when $x$ and $\pi$ are fixed.
Note that the parts of the objective of \eqref{2S-C} that are constant for fixed $x$ and $\pi$, namely $\sum_{i = 1}^n C_i x_i + \Gamma \pi$, are omitted in this definition.
For all $x \in \{0, 1\}^n$ with $\sum_{i = 1}^n x_i \le p$ and all $\pi \in \Rge$, we set:
\begin{align*} \tag{$\triangle$} \label{gxpi}
  g(x,\pi) := \min~ & \sum_{i \in R_x} \underline{c}_i y_i + \sum_{i \in R_x} \rho_i \\
  \st~ & \sum_{i \in R_x} y_i = p_x \\
                    & \pi + \rho_i \ge d_i y_i \\
                    & y \in [0, 1]^{R_x}, \rho \in \Rge^n
\end{align*}
Then the following problem formulation is equivalent to problem \eqref{2S-C}:
\begin{align*} \tag{2S-C'} \label{2S-C'}
  \min~ & \sum_{i = 1}^n C_i x_i + \Gamma \pi + g(x, \pi)\\
  \st~ & \sum_{i = 1}^n x_i \le p \\
        & x \in \{0, 1\}^n, \pi \in \Rge
\end{align*}

Similarly to Section~\ref{sec_poly_alg}, we can rewrite \eqref{gxpi} as follows.
First observe that the optimal value for $\rho$ satisfies $\rho_i = \max\{0, d_i y_i - \pi\}$ for all $i \in [n]$.
Splitting each variable $y_i$ into separate variables $\underline{y}_i$ and $\overline{y}_i$,
representing the two cases where $d_i y_i - \pi \le 0$ and $d_i y_i - \pi > 0$,
results in the following formulation
(see also \eqref{fjpi'} and Lemma~\ref{lem_fjpi_fjpi'}):
\begin{align*}
  g(x, \pi) = \min~ & \sum_{i \in R_x} \underline{c}_i \underline{y}_i + \sum_{i \in R_x} \overline{c}_i \overline{y}_i \\
  \st~ & \sum_{i \in R_x} (\underline{y}_i + \overline{y}_i) = p_x \\
                    & \underline{y}_i \in [0, \pi_i] & \forall i \in R_x \\
                    & \overline{y}_i \in [0, 1 - \pi_i] & \forall i \in R_x
\end{align*}
Here, we again set $\pi_i := \min\{1, \pi/d_i\}$ for all $i \in [n]$. In the construction that we use later, we ensure that $d_i>0$ for all $i \in [n]$, which means that the term is always well-defined.

In the next step, we scale the ranges of the variables $\underline{y}$ and $\overline{y}$,
like in \eqref{fjpi''} and Lemma~\ref{lem_fjpi'_fjpi''},
and get a continuous min-knapsack problem:
\begin{align*} \tag{$\triangle$'} \label{CKP}
  g(x, \pi) = \min~ & \sum_{i \in R_x} \underline{c}_i \pi_i \underline{z}_i + \sum_{i \in R_x} \overline{c}_i (1 - \pi_i) \overline{z}_i \\
  \st~ & \sum_{i \in R_x} (\pi_i \underline{z}_i + (1 - \pi_i) \overline{z}_i) = p_x \\
        & \underline{z}_i, \overline{z}_i \in [0, 1] & \forall i \in R_x
\end{align*}

Our aim is to prove that the problem~\eqref{2S-C'} is NP-hard, both in general and for a fixed value of~$\pi$.
We will use the special structure of the continuous min-knapsack problem \eqref{CKP} in order to understand how its optimal solution changes when selecting different first-stage solutions $x \in \{0, 1\}^n$ and different values~$\pi \in \Rge$.
Like in Section~\ref{sec_poly_alg}, an optimal solution to such a continuous min-knapsack problem can be found by sorting the items by their ratios of cost and size, 
which are again given by $\ratio(\underline z_i) := \underline{c}_i \pi_i / \pi_i = \underline{c}_i$ and $\ratio(\overline z_i) := \overline{c}_i (1 - \pi_i) / (1 - \pi_i) = \overline{c}_i$ here, 
and then selecting items in this order until the capacity $p_x$ is reached. 
Note that this order of all $2n$ items (corresponding to $\underline{z}_i$ and $\overline{z}_i$ for all $i \in [n]$) is independent of $x$ and of $\pi$. 
Different first-stage solutions $x$ lead to different subsets of these $2n$ items being considered (corresponding to $R_x$) and different capacities $p_x$. 
Different values of $\pi$ lead to different sizes $\pi_i$ and $1 - \pi_i$ of the items, and corresponding costs $\underline{c}_i \pi_i$ and $\overline{c}_i (1 - \pi_i)$.

The following result describes the behavior of the continuous min-knapsack problem~\eqref{CKP} for varying values of $\pi$,
i.e., the behavior of the function $g(x,\cdot) \colon \Rge \to \Rge$ in $\pi$.
It is a generalization of Lemma~\ref{lem_fj}.
Note that the only difference between the continuous min-knapsack problems~\eqref{fjpi''} and~\eqref{CKP} concerns the knapsack's capacity.

\begin{lemma} \label{lem_fj_larger_p}
  For every $x \in \{0, 1\}^n$ with $\sum_{i = 1}^n x_i \le p$, the function $g(x,\cdot)$ is continuous, convex, and piecewise linear.
\end{lemma}

\begin{proof}
  Consider the definition of the function $g(x,\cdot)$ in \eqref{gxpi}.
  Analogously to the proof of Lemma~\ref{lem_fj}, this problem is a bounded linear program for each fixed~$\pi$, where $\pi$ appears in the right-hand side vector. Therefore, by general linear programming arguments (see, e.g., \cite[Section~5.2]{bertsimas1997}), $g(x,\cdot)$ is a piecewise linear and convex function in $\pi$.
\end{proof}

The function $g(x, \cdot)$ can be easily transformed into a function $h_x : \Rge \to \Rge$ that describes the objective values of \eqref{2S-C'} for the fixed $x$ and varying values of $\pi$.
Indeed, we only need to add the term $\sum_{i = 1}^n C_i x_i$, which is constant for fixed $x$, and the linear function $\Gamma \pi$, i.e., we set, for all $\pi \in \Rge$:
\begin{align*}
  h_x(\pi) := \sum_{i = 1}^n C_i x_i + \Gamma \pi + g(x, \pi)
\end{align*}
Clearly, the structural properties of the function $g(x, \cdot)$ are preserved by these operations.
Thus, we have:

\begin{lemma} \label{lem_hx}
  For every $x \in \{0, 1\}^n$ with $\sum_{i = 1}^n x_i \le p$, the function $h_x$ is continuous, convex, and piecewise linear.
\end{lemma}

In the proof below, we will make use of these properties when identifying the minimum of such a function $h_x$.
More specifically, we will show that the minimum is attained for some specific~$\pi^\star$ by analyzing the slopes of the linear pieces left and right of $\pi^\star$. If the slope on the left is negative and the slope on the right is positive, then we have found the unique minimum due to the convexity.
Identifying the minima of the functions $h_x$ is important because the optimal value of \eqref{2S-C'} corresponds to the minimum, over all feasible $x$, of these minima.

Based on these general structural observations, we now present our hardness result:

\begin{theorem} \label{thm_hardness}
  The problem (2S-C) is NP-hard.
\end{theorem}

In order to prove Theorem~\ref{thm_hardness},
we give a reduction from the NP-hard (binary) knapsack problem (see~\cite{karp1972})
to the equivalent reformulation~\eqref{2S-C'} of the problem (2S-C).
Let an instance $a, v \in \N^m, b \in \N$ of the knapsack problem be given.
The goal is to select a subset $S \subseteq [m]$ of the items such that their total size fits in the knapsack, i.e., $\sum_{i \in S} a_i \le b$, and their total value $\sum_{i \in S} v_i$ is maximized.
We can assume without loss of generality that $a_i \le b - 1$ for all $i \in [m]$.
Moreover, we assume that $\sum_{i = 1}^m a_i$ is divisible by $b$ (which is without loss of generality, as we can always add one dummy item with value $0$ for this purpose).

We scale the sizes of the items and of the knapsack by the factor $1/b$ and define $a_i' := a_i/b$ for all $i \in [m]$. Then the original constraint $\sum_{i \in S} a_i \le b$ is equivalent to $\sum_{i \in S} a_i' \le 1$.
Our assumptions imply that $a_i' \le 1 - 1/b$ for all $i \in [m]$ and that $A' := \sum_{i = 1}^m a_i' < m$ is an integer.
Moreover, set $V := \sum_{i = 1}^m v_i$.

We now define an instance of the problem~\eqref{2S-C'} as follows.
It consists of $n := 2 m - A' + 3$ items: one for each knapsack item, one additional ``expensive'' item, and $m - A' + 2$ additional ``cheap'' items.
Let $M := 3 m b V$ be a large constant.
The items' costs are defined by
\begin{align*}
  & C_i := M a_i', && \underline{c}_i := v_i/(1 - a_i'), && d_i := M/(1 - a_i'), && \text{for } i \in \{1, \dots, m\}, \\
  & C_i := (m + 3) M, && \underline{c}_i := M, && d_i := M, && \text{for } i = m + 1, \\
  & C_i := (m + 3) M, && \underline{c}_i := 0, && d_i := M, && \text{for } i \in \{m + 2, \dots, n\}.
\end{align*}
Finally, set the number of items to select as $p := 2 m - 2 A' + 3$
and the uncertainty budget as $\Gamma := m - A' + 1$.

To ease the following discussion, we illustrate this construction with some specific values.

\begin{example}
We consider a knapsack instance with $m=4$ items, for which the item sizes and values are given in Table~\ref{tab:ex1}. We use $b=7$. Additionally, we show all possible solution candidates $S^1,\dots,S^{16}$ for this knapsack instance, along with their sizes and values. Only some of these candidates are indeed feasible, which is marked in the column ``feasible?''.

\begin{table}
  \begin{center}
    \begin{tabular}{rrrrrrrrr}
      \toprule
      &\multicolumn{4}{c}{item $i$} \\
      \cmidrule{2-5}
      & 1 & 2 & 3 & 4 & $\sum_{i\in S} a_i$ & $\sum_{i\in S} v_i$ & feasible? & robust value \\
      \midrule
      $a_i$ & 2 & 3 & 4 & 5 & & & \\
      $a'_i$ & $2/7$ & $3/7$ & $4/7$ & $5/7$ & & & \\
      $v_i$ & 3 & 6 & 7 & 9 & & & \\
      \midrule
      $S^1$  & 0 & 0 & 0 & 0 & 0 & 0 & \checkmark & 8425.00 \\
      $S^2$  & 0 & 0 & 0 & 1 & 5 & 9 & \checkmark & 8416.00 \\
      $S^3$  & 0 & 0 & 1 & 0 & 4 & 7 & \checkmark & 8418.00 \\
      $S^4$  & 0 & 0 & 1 & 1 & 9 & 16 & & 8667.97 \\
      $S^5$  & 0 & 1 & 0 & 0 & 3 & 6 & \checkmark & 8419.00 \\
      $S^6$  & 0 & 1 & 0 & 1 & 8 & 15 & & 8534.72 \\
      $S^7$  & 0 & 1 & 1 & 0 & 7 & 13 & \checkmark & 8412.00 \\
      $S^8$  & 0 & 1 & 1 & 1 & 12 & 22 & & 8948.00 \\
      $S^9$  & 1 & 0 & 0 & 0 & 2 & 3 & \checkmark & 8422.00 \\
      $S^{10}$  & 1 & 0 & 0 & 1 & 7 & 12 & \checkmark & 8413.00\\
      $S^{11}$  & 1 & 0 & 1 & 0 & 6 & 10 & \checkmark & 8415.00\\
      $S^{12}$  & 1 & 0 & 1 & 1 & 11 & 19 & & 8817.75\\
      $S^{13}$  & 1 & 1 & 0 & 0 & 5 & 9 & \checkmark & 8416.00\\
      $S^{14}$  & 1 & 1 & 0 & 1 & 10 & 18 & & 8696.65\\
      $S^{15}$  & 1 & 1 & 1 & 0  & 9 & 16 & & 8588.40\\
      $S^{16}$  & 1 & 1 & 1 & 1 & 14 &  25 & & 8925.00 \\
      \bottomrule
    \end{tabular}
  \end{center}
  \caption{Example knapsack instance and solution candidates.}\label{tab:ex1}
\end{table}

In Table~\ref{tab:ex2}, we show the two-stage selection instance that is constructed in our reduction. By our definitions, we obtain $A'=2$, $V=25$, $n=9$, $p=7$, $\Gamma=3$, and $M=2100$.

\begin{table}
  \begin{center}
    \begin{tabular}{rrrrrrrrrrrr}
      \toprule
      &\multicolumn{4}{c}{knapsack items}&& \multicolumn{1}{c}{expensive item} && \multicolumn{4}{c}{cheap items} \\
      \cmidrule{2-5}
      \cmidrule{7-7}
      \cmidrule{9-12}
      $i$ & 1 & 2 & 3 & 4 & ~ & 5 & ~ & 6 & 7 & 8 & 9 \\
      \midrule
      $C_i$ & 600 & 900 & 1200 & 1500 & & 14700 & & 14700& 14700 & 14700 & 14700 \\
      $\underline{c}_i$ & $4.2$ & $10.5$ & $16.3$ & $31.5$ & & 2100 & & 0 & 0 & 0 & 0 \\
      $d_i$ & 2940 & 3675 & 4900 & 7350 & & 2100 & & 2100 & 2100 & 2100 & 2100 \\
      \bottomrule
    \end{tabular}
  \end{center}
  \caption{Example two-stage selection instance constructed in our reduction.}\label{tab:ex2}
\end{table}

Observe that selecting any of the last five items in the first stage immediately results in high costs. For each of the solution candidates that only pick first-stage items amongst the first four items, there is a corresponding knapsack solution candidate. In Table~\ref{tab:ex1}, the final column labeled ``robust value'' shows the objective value of this solution for the two-stage selection problem. We can observe that knapsack solution candidates that are infeasible always result in higher objective values than solution candidates that are feasible. Furthermore, we note that the best knapsack solution ($S^7$, with value 13) corresponds to the best two-stage selection solution (with value 8412). Indeed, we show below that, for feasible knapsack solutions $x$, the objective value of the corresponding solution of the constructed two-stage selection instance is $(\Gamma + 1) M + V - \sum_{i = 1}^m v_i x_i = 8425 - \sum_{i = 1}^m v_i x_i$.
\end{example}

With this example in mind, we now return to the proof of Theorem~\ref{thm_hardness}.
We know from Lemma~\ref{lem_hx} that, for every fixed feasible $x$, the objective value of~\eqref{2S-C'} is a continuous, convex, and piecewise linear function~$h_x$ in $\pi$. The optimal value of~\eqref{2S-C'} can be written as the minimum of the minima of all these functions, i.e., as
\begin{align*}
  \min_{x \in \{0, 1\}^n : \sum_{i = 1}^n x_i \le p} ~\min_{\pi \in \Rge} h_x(\pi).
\end{align*}
In the following, we will analyze the functions $h_x$, first for $x$ that correspond to feasible knapsack solutions (Lemma~\ref{lem_hardness_feasible}) and second for $x$ that do not correspond to feasible knapsack solutions (Lemma~\ref{lem_hardness_infeasible}). In the former case, we will show that the minimal value of $h_x$ is negatively related to the knapsack solution's value. In the latter case, we will show that the minimal value of $h_x$ is always larger than the minimal values occurring in the first case.
Together, this proves that the optimum of~\eqref{2S-C'} corresponds to an optimal knapsack solution, finishing our reduction.

\begin{lemma} \label{lem_hardness_feasible}
  Let $x \in \{0, 1\}^n$ be a vector such that $x_1,\dots,x_m$ encodes a feasible solution of the given knapsack instance, i.e., such that $\sum_{i = 1}^m a_i' x_i \le 1$,
  and $x_i = 0$ for all $i \in \{m + 1, \dots, n\}$.
  Then, the following holds:
  \begin{enumerate}[label=(\alph*)]
  \item $x$ is feasible for problem~\eqref{2S-C'}, i.e., $\sum_{i = 1}^m x_i \le p$. \label{lem_hardness_feasible_feasible}
  \item When fixing $x$ in \eqref{2S-C'}, the unique optimal value for $\pi$ is $\pi = M$, \\ i.e., $\argmin_{\pi \in \Rge} h_x(\pi) = \{M\}$. \label{lem_hardness_feasible_opt_pi}
  \item When fixing $x$ in \eqref{2S-C'}, the optimal objective value is given by $\min_{\pi \in \Rge} h_x(\pi) = {(\Gamma + 1) M + V - \sum_{i = 1}^m v_i x_i}$. \label{lem_hardness_feasible_opt_cost}
  \end{enumerate}
\end{lemma}

\begin{proof}
  We first show \ref{lem_hardness_feasible_feasible}. Due to $\sum_{i = 1}^m a_i' x_i \le 1$ and the assumption that $a_i' < 1$ holds for all $i \in [m]$, we have
  \begin{align*}
    \sum_{i = 1}^m x_i + A' = \sum_{i = 1}^m x_i + \sum_{i = 1}^m a_i' x_i + \sum_{i = 1}^m a_i' (1 - x_i) \le \sum_{i = 1}^m x_i + 1
    + \sum_{i = 1}^m (1 - x_i) = m + 1,
  \end{align*}
  and therefore $\sum_{i = 1}^m x_i \le m - A' + 1 < p$.
  Hence, $x$ is feasible for problem~\eqref{2S-C'}.

  Next, we prove that $h_x(M) = (\Gamma + 1) M + V - \sum_{i = 1}^m v_i x_i$.
  After that, we will argue that this is in fact the unique minimum of $h_x$, which then implies \ref{lem_hardness_feasible_opt_pi} and \ref{lem_hardness_feasible_opt_cost}.

  In order to determine $h_x(M)$, we study the behavior of the continuous min-knapsack problem~\eqref{CKP} for our fixed $x$ and the fixed value $\pi = M$.
  As explained above, an optimal solution of this problem can be described by sorting its items by the values $\underline{c}$ and $\overline{c}$. Note that, if the sorting is not unique due to several of the values being equal, any order of them works.
  Here, the items corresponding to $\underline{c}_i = 0$ for $i \in \{m + 2, \dots, n\}$ come first in this order, followed by the items corresponding to $\underline{c}_i$ for $i \in [m] \cap R_x$ (and we are not interested in their precise order). Note that, for all $i \in [m]$, we have $0 \le \underline{c}_i = v_i/(1 - a_i') \le b v_i < M$. Next are the items corresponding to $\underline{c}_{m + 1} = M$ and $\overline{c}_i = M$ for $i \in \{m + 2, \dots, n\}$. Finally, the items corresponding to $\overline{c}_i > M$ for $i \in [m + 1] \cap R_x$ follow (again, we are not interested in their precise order).
  We will see that the items corresponding to $\overline{c}_i$ are never selected.
  Regarding the sizes of the items in~\eqref{CKP},
  note that we have $\pi \le d_i$ for all $i \in [n]$, and therefore $\pi_i = \pi/d_i = M/(M/(1 - a_i')) = 1 - a_i'$ for all $i \in [m]$ and $\pi_{i} = M/M = 1$ for all $i \in \{m + 1, \dots, n\}$.
  We now investigate when the capacity $p_x$ is reached when adding the sizes ($\pi_i$ or $(1 - \pi_i)$) of the items in the given order.
  The items corresponding to $\underline{c}_i$ for $i \in R_x \setminus \{m + 1\}$ have a total size of
  \begin{align*}
    \sum_{i \in R_x \setminus \{m + 1\}} \pi_i &= (m - A' + 2) + \sum_{i = 1}^m \pi_i (1 - x_i) = (m - A' + 2) + \sum_{i = 1}^m (1 - a_i') (1 - x_i) \\
                                               &= (m - A' + 2) + m - \sum_{i = 1}^m a_i' - \sum_{i = 1}^m x_i + \sum_{i = 1}^m a_i' x_i = 2 m - 2 A' + 2 - \sum_{i = 1}^m x_i + \sum_{i = 1}^m a_i' x_i \\
                                               &= p - 1 - \sum_{i = 1}^m x_i + \sum_{i = 1}^m a_i' x_i = p_x + \sum_{i = 1}^m a_i' x_i - 1 \le p_x.
  \end{align*}
  Hence, all of these items are packed, i.e., $\underline{z}_i = 1$ for all $i \in R_x \setminus \{m + 1\}$.
  It remains to select a fraction of $1 - \sum_{i = 1}^m a_i' x_i \ge 0$ of the next item in the order,
  which corresponds to $\underline{c}_{m + 1}$ and has size $\pi_{m + 1} = 1$.
  Hence, we set $\underline{z}_{m + 1} = 1 - \sum_{i = 1}^m a_i' x_i$.
  Therefore, the optimal value of~\eqref{CKP} can be written as
  \begin{align*}
    g(x, M) = \sum_{i \in R_x} \underline{c}_i \pi_i \underline{z}_i = \sum_{i = 1}^m \underline{c}_i \pi_i (1 - x_i) + \underline{c}_{m + 1} \pi_{m + 1} \underline{z}_{m + 1} = \sum_{i = 1}^m v_i (1 - x_i) + M (1 - \sum_{i = 1}^m a_i' x_i).
  \end{align*}
  Together with the first-stage costs $\sum_{i = 1}^m C_i x_i = M \sum_{i = 1}^m a_i' x_i$,
  the total objective value of the solution consisting of $x$ and $\pi = M$ in~\eqref{2S-C'} is
  \begin{align*}
    h_x(M) = \sum_{i = 1}^m C_i x_i + \Gamma M + g(x, M) = (\Gamma + 1) M + V - \sum_{i = 1}^m v_i x_i.
  \end{align*}

  It remains to show that this is the unique minimum of the function $h_x$.
  For this, recall from Lemma~\ref{lem_hx} that $h_x$ is continuous, convex, and piecewise linear.
  We will show that $\pi = M$ is a breakpoint of $h_x$, the linear piece left of $\pi = M$ has negative slope, and the linear piece right of $\pi = M$ has positive slope.
  As $h_x$ is convex, this implies that $\pi = M$ is the unique minimum.

  Recall that changing $\pi$ changes the item sizes and costs in the continuous min-knapsack problem~\eqref{CKP}, while the ordering of the items remains the same.
  Let us first consider the situation where $\pi$ is slightly smaller than $M$.
  The intuition is that all of the items that are packed in the solution described above, corresponding to $\underline{c}_i$ for $i \in R_x$, become slightly smaller as well, which leads to a slightly larger fraction $\underline{z}_{m + 1}$ of the expensive fractional item being packed, thus increasing the total objective value.
  More precisely, when decreasing $\pi$, the size $\pi_i = \pi/d_i$ of each item corresponding to some $\underline{c}_i$ decreases linearly at a rate of $1/d_i$. The total size of the fully packed items, corresponding to $\underline{c}_i$ for all $i \in R_x \setminus \{m + 1\}$, thus decreases at a rate of $\sum_{i \in R_x \setminus \{m + 1\}} 1/d_i$.
  Accordingly, the size of the additionally packed fraction of the item corresponding to $\underline{c}_{m + 1} = M$ increases at the same rate.
  The total cost of the solution then changes linearly as well, at a rate of
  \begin{align*}
    &- \sum_{i \in R_x \setminus \{m + 1\}} \frac{\underline{c}_i}{d_i} + \underline{c}_{m + 1} \sum_{i \in R_x \setminus \{m + 1\}} \frac{1}{d_i} = \sum_{i \in R_x \setminus \{m + 1\}} \frac{\underline{c}_{m + 1} - \underline{c}_i}{d_i} \\
    &= (m - A' + 2) \frac{M - 0}{M} + \sum_{i = 1}^m (1 - x_i) \frac{M - v_i/(1 - a_i')}{M/(1 - a_i')} \\
    &= \Gamma + 1 + \sum_{i = 1}^m (1 - x_i)(1 - a_i') - \sum_{i = 1}^m (1 - x_i) \frac{v_i}{M} \\
    &\ge \Gamma + 1 - \frac{V}{M} > \Gamma.
  \end{align*}
  Hence, the slope of the linear piece left of $\pi = M$ is less than $- \Gamma$ for the function $g(x, \cdot)$
  and negative for the function $h_x$, which includes the additional linear function $\Gamma \pi$.

  Similarly, consider the situation where $\pi$ is slightly larger than $M$. Here, the items corresponding to $\underline{c}_i$ for $i \in [m] \cap R_x$ become slightly larger, while the size of the cheap items $i \in \{m + 2, \dots, n\}$ is constant $\pi_i = \min\{1, \pi/M\} = 1$ now. This leads to a slightly smaller fraction $\underline{z}_{m + 1}$ of the expensive fractional item being packed or, if $\underline{z}_{m + 1} = 0$ holds already for $\pi = M$ (i.e., if $\sum_{i = 1}^m a_i' x_i = 1$), to some small fraction of the previous item in the order being unpacked. This previous item corresponds to $\max_{i \in [m] \cap R_x} \underline{c}_i < M$.
  Similarly as above, the change in the total cost of the solution is linear again and happens at a rate of at least
  \begin{align*}
    &\sum_{i \in [m] \cap R_x} \frac{\underline{c}_i}{d_i} - M \sum_{i \in [m] \cap R_x} \frac{1}{d_i} = \sum_{i = 1}^m (1 - x_i) \frac{\underline{c}_i - M}{d_i} \\
    &= \sum_{i = 1}^m (1 - x_i) \frac{v_i/(1 - a_i') - M}{M/(1 - a_i')} = \sum_{i = 1}^m (1 - x_i) \frac{v_i}{M} - \sum_{i = 1}^m (1 - x_i) (1 - a_i') \\
    &= \sum_{i = 1}^m (1 - x_i) \frac{v_i}{M} - m + \sum_{i = 1}^m x_i + \sum_{i = 1}^m a_i' - \sum_{i = 1}^m a_i' x_i > -m + A' - 1 = - \Gamma.
  \end{align*}
  Hence, the slope of the linear piece right of $\pi = M$ is greater than $- \Gamma$ for the function $g(x, \cdot)$
  and positive for the function $h_x$.
  
  This shows that the convex function $h_x$ indeed attains its unique minimum in $\pi = M$ and thus finishes the proof of \ref{lem_hardness_feasible_opt_pi} and \ref{lem_hardness_feasible_opt_cost}.
\end{proof}

\begin{lemma} \label{lem_hardness_infeasible}
  Let $x \in \{0, 1\}^n$ be such that it is feasible for problem~\eqref{2S-C'}, i.e., $\sum_{i = 1}^n x_i \le p$, and does not meet the conditions of Lemma~\ref{lem_hardness_feasible}, i.e., $\sum_{i = 1}^m a_i' x_i > 1$ or $x_i = 1$ for some $i \in \{m + 1, \dots, n\}$.
  Then, when fixing $x$ in \eqref{2S-C'}, the optimal objective value satisfies $\min_{\pi \in \Rge} h_x(\pi) > (\Gamma + 1) M + V$.
\end{lemma}

\begin{proof}
  First consider the case where $x_i = 1$ for some $i \in \{m + 1, \dots, n\}$.
  Then the first-stage costs of the solution $x$ are very high and we have $\min_{\pi \in \Rge} h_x(\pi) \ge \sum_{i = 1}^n C_i x_i \ge (m + 3) M > (\Gamma + 2) M > (\Gamma + 1) M + V$, which proves the claim in this case.

  For the rest of the proof, assume that $\sum_{i = 1}^m a_i' x_i > 1$ and $x_i = 0$ for all $i \in \{m + 1, \dots, n\}$.
  Define $\pi^\star = M p_x / (p_x + \sum_{i = 1}^m a_i' x_i - 1)$.
  Similarly to the proof of Lemma~\ref{lem_hardness_feasible}, we will first determine the value $h_x(\pi^\star)$ and show that $h_x(\pi^\star) > (\Gamma + 1) M + V$, and then argue that $\pi = \pi^\star$ is the unique minimum of $h_x$ by investigating the slopes left and right of this point.

  As in the proof of Lemma~\ref{lem_hardness_feasible}, $h_x(\pi^\star)$ can be determined by studying the behavior of the continuous min-knapsack problem~\eqref{CKP} for $x$ and $\pi^\star$. In fact, its optimal solution is given by $\underline{z}_i = 1$ for all $i \in R_x \setminus \{m + 1\}$ because these are again the first items in the greedy order and their total size precisely fills the capacity~$p_x$:
  \begin{align*}
    \sum_{i \in R_x \setminus \{m + 1\}} \pi_i &= (m - A' + 2) \frac{\pi^\star}{M} + \sum_{i = 1}^m (1 - x_i) \frac{\pi^\star}{d_i}\\
                                               &= \frac{p_x}{p_x + \sum_{i = 1}^m a_i' x_i - 1} \left(m - A' + 2 + \sum_{i = 1}^m (1 - x_i)(1 - a_i')\right)  \\
                                               &= \frac{p_x}{p_x + \sum_{i = 1}^m a_i' x_i - 1} \left(p_x + \sum_{i = 1}^m a_i' x_i - 1\right) = p_x
  \end{align*}
  The optimal value of \eqref{CKP} is therefore
  \begin{align*}
    g(x, \pi^\star) = \sum_{i \in R_x} \underline{c}_i \pi_i \underline{z}_i = \sum_{i = 1}^m (1 - x_i) \underline{c}_i \frac{\pi^\star}{d_i} = \sum_{i = 1}^m v_i (1 - x_i) \frac{\pi^\star}{M} = \left(V - \sum_{i = 1}^m v_i x_i\right) \frac{\pi^\star}{M}.
  \end{align*}
  The total objective value of $(x, \pi^\star)$ in \eqref{2S-C'} can then be written as
  \begin{align*}
    h_x(\pi^\star) &= \sum_{i = 1}^m C_i x_i + \Gamma \pi^\star + g(x, \pi^\star) = M \sum_{i = 1}^m a_i' x_i + \Gamma \pi^\star + \left(V - \sum_{i = 1}^m v_i x_i\right) \frac{\pi^\star}{M} \\
               &= M + M \left(\sum_{i = 1}^m a_i' x_i - 1\right) + \Gamma M - \Gamma M \left(1 - \frac{\pi^\star}{M}\right) + V - V \left(1 - \frac{\pi^\star}{M}\right) - \sum_{i = 1}^m v_i x_i \frac{\pi^\star}{M}\\
               &= (\Gamma + 1) M + V + \varepsilon,
  \end{align*}
  where $\varepsilon = M (\sum_{i = 1}^m a_i' x_i - 1) - \Gamma M (1 - \pi^\star/M) - V (1 - \pi^\star/M) - \sum_{i = 1}^m v_i x_i \pi^\star/M$.
  It remains to show that $\varepsilon > 0$ because this implies the claimed assertion $h_x(\pi^\star) > (\Gamma + 1) M + V$.
  We can derive
  \begin{align*}
    \varepsilon &= M \left(\sum_{i = 1}^m a_i' x_i - 1 - \Gamma \left(1 - \frac{\pi^\star}{M}\right)\right) - V + \sum_{i = 1}^m v_i (1 - x_i) \frac{\pi^\star}{M} \\
                &\ge M \left(\sum_{i = 1}^m a_i' x_i - 1 - \Gamma \frac{\sum_{i = 1}^m a_i' x_i - 1}{p_x + \sum_{i = 1}^m a_i' x_i - 1}\right) - V \\
                &= M \left(\sum_{i = 1}^m a_i' x_i - 1\right) \delta - V ,
  \end{align*}
  where $\delta := 1 - \Gamma/(p_x + \sum_{i = 1}^m a_i' x_i - 1)$.
  Note that we have $\delta > 1/(3m)$ because $p_x + \sum_{i = 1}^m a_i' x_i - 1 < p + A' = 2m - A' + 3 \le 3m$ and
  \begin{align*}
    p_x + \sum_{i = 1}^m a_i' x_i - 1 - \Gamma &= m - A' + 1 - \sum_{i = 1}^m x_i + \sum_{i = 1}^m a_i' x_i \\
                                               &= \sum_{i = 1}^m (1 - a_i')(1 - x_i) + 1 \ge 1.
  \end{align*}
  Together with $\sum_{i = 1}^m a_i' x_i - 1 \ge 1/b$ (which follows from $\sum_{i = 1}^m a_i' x_i > 1$ due to the integrality of $a$ and $b$) and the definition of $M$, this finally gives $\varepsilon > M/(3mb) - V = 0$.

  Next, we prove that the function $h_x$ actually attains its minimum at $\pi = \pi^\star$.
  The slope of $h_x$ left of $\pi = \pi^\star$ is exactly the same as in the proof of Lemma~\ref{lem_hardness_feasible} and therefore again negative.
  In the analysis of the slope right of $\pi = \pi^\star$, there are two differences to the proof of Lemma~\ref{lem_hardness_feasible}.
  First, since $\pi^\star < M$, not only the sizes of the items corresponding to $\underline{c}_i$ for $i \in [m] \cap R_x$ increase, but also the sizes $\pi_i = \pi/M$ of the cheap items corresponding to $\underline{c}_i$ for $i \in \{m + 2, \dots, n\}$.
  Second, we have $\underline{z}_{m + 1} = 0$ already at $\pi = \pi^\star$, so the item we unpack a fraction of is always the previous one in the order, corresponding to $\underline{c}^\star := \max_{i \in [m] \cap R_x} \underline{c}_i$. Note that $\underline{c}_i = v_i/(1 - a_i') \le bV$ for all $i \in [m]$ and therefore $\underline{c}^\star \le bV$. The total cost of the solution of \eqref{CKP}, when slightly increasing $\pi$ starting from $\pi = \pi^\star$, now changes linearly at a rate of
  \begin{align*}
    \sum_{i \in [m] \cap R_x} \frac{\underline{c}_i}{d_i} - \underline{c}^\star \sum_{i \in R_x \setminus \{m + 1\}} \frac{1}{d_i} &\ge - \underline{c}^\star \left(\sum_{i = 1}^m (1 - x_i) \frac{1}{d_i} + (m - A' + 2)\frac{1}{M}\right) \\
                                                                                                                 &= - \frac{\underline{c}^\star}{M} \left(\sum_{i = 1}^m (1 - x_i) (1 - a_i') + m - A' + 2\right) \\
                                                                                                                 &= - \frac{\underline{c}^\star}{M} \left(2 \Gamma - \sum_{i = 1}^m x_i (1 - a_i')\right) \\
                                                                                                                 &\ge - \Gamma \frac{2 \underline{c}^\star}{M} \ge - \Gamma \frac{2 b V}{M} > - \Gamma
  \end{align*}
  Hence, like in the proof of Lemma~\ref{lem_hardness_feasible}, the slope of the linear piece right of $\pi = \pi^\star$ is greater than $- \Gamma$ for the function $g(x, \cdot)$ and positive for the function $h_x$.
  This shows that $\pi^\star$ indeed represents the unique minimum of $h_x$ and thus finishes the proof.
\end{proof}

As outlined above, we have now presented all ingredients for the proof of Theorem~\ref{thm_hardness}:

\begin{proof}[Proof of Theorem~\ref{thm_hardness}]
  We have described a reduction from the NP-hard knapsack problem to the problem~\eqref{2S-C'}.
  In Lemmas~\ref{lem_hardness_feasible} and~\ref{lem_hardness_infeasible}, we have analyzed two types of feasible solutions~$x$ of the constructed instance of~\eqref{2S-C'}.
  The cost of the solutions considered in Lemma~\ref{lem_hardness_infeasible} is always larger than the cost of the solutions considered in Lemma~\ref{lem_hardness_feasible}.
  The solutions considered in Lemma~\ref{lem_hardness_feasible} directly correspond to the feasible solutions of the given knpasack instance.
  Moreover, according to Lemma~\ref{lem_hardness_feasible}\ref{lem_hardness_feasible_opt_cost}, the objective value achieved by such a solution~$x$, together with the best corresponding choice of~$\pi$, is minimal when the value of the knapsack solution is maximal.
  Thus, an optimal solution of problem~\eqref{2S-C'} directly gives an optimal solution of the knapsack problem.
\end{proof}

Our reduction can also be used to conclude that problem~\eqref{2S-C'} is still NP-hard when the value of $\pi$ is fixed:

\begin{theorem} \label{thm_hardness_fixed_pi}
  The problem~\eqref{2S-C'} with a fixed value of $\pi$ is NP-hard.
\end{theorem}

\begin{proof}
  A reduction from the NP-hard knapsack problem can be obtained using the same construction as for Theorem~\ref{thm_hardness} and additionally setting the fixed value of $\pi$ as $\pi := M$.
  Due to Lemma~\ref{lem_hardness_feasible}\ref{lem_hardness_feasible_opt_pi}, the solutions~$x$ of the constructed instance of~\eqref{2S-C'} that correspond to feasible knapsack solutions behave exactly as before.
  Lemma~\ref{lem_hardness_infeasible} implies that, for all other feasible solutions~$x$ of~\eqref{2S-C'}, the cost $h_x(M) \ge \min_{\pi \in \Rge} h_x(\pi) > (\Gamma + 1) M + V$ is again worse than the one of any solution of the first kind.
  Thus, also when fixing $\pi = M$, an optimal solution of problem~\eqref{2S-C'} gives an optimal solution of the knapsack problem.
\end{proof}

\begin{remark}
  While Theorem~\ref{thm_hardness} implies NP-hardness of (2S-C),
  it is a priori not clear whether the decision variant of (2S-C) is contained in NP.
  However, the containment in NP follows from the arguments presented in \cite[Corollary 4]{goerigk2024complexity}.
  Therefore, the decision variant of (2S-C) is NP-complete.
\end{remark}

Finally, we demonstrate that the NP-hardness result of Theorem~\ref{thm_hardness} can be extended from the two-stage robust selection problem to two-stage versions of other nominal problems, in particular the assignment problem. This partially answers the open problem number 15 in \cite{robook}.
In the assignment problem, a bipartite graph with edge costs is given, and the task is to find a perfect matching of minimum cost.

\begin{corollary} \label{cor_hardness_assignment}
  The two-stage robust assignment problem with continuous budgeted uncertainty is NP-hard.
\end{corollary}

\begin{proof}
  The assignment problem can be considered to be a special case of the selection problem; see, e.g., \cite{KasperskiEtAl2013}.
  Such a construction can easily be transferred to the two-stage versions of the problems:
  Let an instance $(n, p, C, \underline{c}, d, \Gamma)$ of the two-stage selection problem with budgeted uncertainty be given.
  Define a bipartite graph $G = (V \cup W, E)$ with $\lvert V \rvert = \lvert W \rvert = 2n - p$,
  and denote its vertices by $V = \{v_1, \dots, v_{2n - p}\}$ and $W = \{w_1, \dots, w_{2n - p}\}$.
  The edge set is given by
  \begin{align*}
    E = \left\{\{v_i, w_i\} : i \in [n]\right\} &\cup \left\{\{v_i, w_j\} : i \in [n], j \in \{n + 1, \dots, 2 n - p\}\right\} \\
    & \cup \left\{\{v_i, w_j\} : i \in \{n + 1, \dots, 2 n - p\}, j \in [n]\right\}.
  \end{align*}
  Define the two-stage assignment problem's first-stage costs $C' \in \Rge^E$ and the parameters $\underline{c}' \in \Rge^E$ and $d' \in \Rge^E$ of the uncertainty set as follows: For all $i \in [n]$, set $C'_{\{v_i, w_i\}} := C_i$, $\underline{c}'_{\{v_i, w_i\}} := \underline{c}_i$, and $d'_{\{v_i, w_i\}} := d_i$. For all other edges $e \in E$, set $C'_e = \underline{c}'_e = d'_e := 0$.
  Finally, set $\Gamma' := \Gamma$.

  It can be easily checked that every perfect matching in $G$ contains exactly $p$ edges from the set $\{\{v_i, w_i\} : i \in [n]\}$.
  Hence, there is a direct correspondence between the feasible item selections and the perfect matchings.
  Finally, note that we may assume that a first-stage solution of the two-stage assignment instance only consists of edges in $\{\{v_i, w_i\} : i \in [n]\}$, since all other edges can still be selected at cost $0$ in the second stage.
\end{proof}

\section{Multi-Representative Selection with Alternative Continuous Budgeted Sets}
\label{sec:multi-alt}

We show that the problem~\eqref{2MRS-AC} can be solved in polynomial time, using similar arguments as in \cite{ChasseinEtAl2018}. First, we substitute the variables $y_i$. Similarly to the reformulations in Section~\ref{sec_poly_alg}, we replace their low-cost part where $y_i\in[0,\pi]$ using variables $\underline{z}_i$, and their high-cost part where $y_i\in[\pi,1]$ using variables $\overline{z}_i$, to obtain the following reformulation:
\begin{align*}
\min\ & \sum_{i=1}^n C_i x_i + \sum_{i=1}^n \underline{c}_i \pi \underline{z}_i + \sum_{i=1}^n \overline{c}_i(1-\pi)\overline{z}_i + \Gamma \pi \\
\st\ & \sum_{i\in T_j} (x_i + \pi \underline{z}_i + (1-\pi)\overline{z}_i) = p_j & \forall j\in[m] \\
& x_i + \underline{z}_i \le 1 & \forall i\in[n] \\
& x_i + \overline{z}_i \le 1 & \forall i\in[n] \\
& x \in \{0, 1\}^n, \underline{z}, \overline{z} \in [0, 1]^n, \pi \in [0, 1]
\end{align*}
For each $j\in[m]$, we define the function $f_j \colon [0, 1] \to \Rge$ as follows:
\begin{align*} \tag{$\circ$} \label{mrs-fprob}
f_j(\pi) = \min\ & \sum_{i\in T_j} C_i x_i + \sum_{i\in T_j} \underline{c}_i \pi \underline{z}_i + \sum_{i\in T_j} \overline{c}_i(1-\pi)\overline{z}_i \\
\st\ & \sum_{i\in T_j} (x_i + \pi \underline{z}_i + (1-\pi)\overline{z}_i) = p_j \\
& x_i + \underline{z}_i \le 1 & \forall i\in T_j \\
& x_i + \overline{z}_i \le 1 & \forall i\in T_j \\
& x \in \{0, 1\}^{T_j}, \underline{z}, \overline{z} \in [0, 1]^{T_j}
\end{align*}
Notice that the optimal value of \eqref{2MRS-AC} is given by $\min_{\pi\in[0,1]} \left(\sum_{j = 1}^m f_j(\pi) + \Gamma \pi\right)$.
The following result is a direct adaptation of \cite[Lemma~6.22]{robook}.

\begin{lemma}\label{lem:mrs-piecewise}
For all $j \in [m]$, the function $f_j$ is piecewise linear and has breakpoints within the set
\[ \Pi_j = [0, 1] \cap \left\{\frac{p_j - a - b}{c-b} : a,b\in\{0,\dots,p_j\}, c\in\{0,\dots,\lvert T_j \rvert\}, a+b \le p_j, c > b \right\}. \]
\end{lemma}

Observe that, if each $f_j$ is piecewise linear with breakpoints in $\Pi_j$, then
$\sum_{j = 1}^m f_j(\pi) + \Gamma \pi$ is also piecewise linear with breakpoints in $\Pi = \bigcup_{j\in[m]} \Pi_j$.
The solution strategy to \eqref{2MRS-AC} is to calculate each value $f_j(\pi)$ for $\pi\in\Pi$, and to choose the best of these solutions. By Lemma~\ref{lem:mrs-piecewise}, we thus ensure that we have found an optimal solution. So let $\pi\in\Pi$ and $j\in[m]$ be fixed.

\begin{lemma}\label{lem:mrs-guess}
  For every $j \in [m]$ and $\pi \in [0, 1]$, there is an optimal solution to the problem~\eqref{mrs-fprob} where at most one of the variables $\bigcup_{i \in T_j} \{\underline{z}_i, \overline{z}_i\}$ is fractional.
\end{lemma}

\begin{proof}
  For every fixed $j \in [m]$, $\pi \in [0, 1]$, and $x \in \{0, 1\}^n$, the problem~\eqref{mrs-fprob} is a linear program that has the structure of a continuous min-knapsack problem, similarly to~\eqref{fjpi''} in Section~\ref{sec_poly_alg} and~\eqref{CKP} in Section~\ref{sec:hardnesscont}.
  Such a problem can be solved by a greedy algorithm~\cite{Dantzig1957}, which gives an optimal solution with at most one fractional variable.
  In particular, this holds when fixing $x$ to the values these variables attain in an optimal solution of~\eqref{mrs-fprob}, which concludes the proof.
\end{proof}

Using the existence of a single fractional item, we can now solve problem~\eqref{mrs-fprob}:

\begin{lemma}\label{lem:mrs-poly}
For every $j \in [m]$ and $\pi \in [0, 1]$, the problem~\eqref{mrs-fprob} can be solved in polynomial time.
\end{lemma}

\begin{proof}
  Using Lemma~\ref{lem:mrs-guess}, we first guess one variable $\underline{z}_i$ or $\overline{z}_i$ that is allowed to be fractional.
  Using a dynamic program, we construct the set of Pareto solutions among
  all binary variables (i.e., all remaining $\underline{z}$ and $\overline{z}$ variables and all $x$ variables)
  with respect to minimum costs and minimum size. As all possible item sizes are 1, $\pi$, or $1-\pi$, this set cannot contain more than $\bigO(\lvert T_j \rvert^3)$ elements. Among these solutions, we then consider all solutions with total size in
  $[p_j - \pi, p_j]$ or $[p_j - (1 - \pi), p_j]$ (if the fractional variable is of type $\underline{z}_i$ or of type $\overline{z}_i$, respectively),
  and extend them to a feasible solution using the fractional variable.
  A solution constructed in this way with minimum costs is then an optimal solution to problem~\eqref{mrs-fprob}.
\end{proof}

\begin{theorem}\label{th:multialtinP}
The problem (2MRS-AC)
can be solved in polynomial time.
\end{theorem}

\begin{proof}
We solve the problem $\min_{\pi\in[0,1]} \left(\sum_{j = 1}^m f_j(\pi) + \Gamma \pi\right)$ by enumerating a polynomial number of breakpoints using Lemma~\ref{lem:mrs-piecewise}. According to Lemma~\ref{lem:mrs-poly}, each value $f_j(\pi)$ can then be calculated in polynomial time, which means that the optimal value is found in polynomial time. Note that we also construct a corresponding feasible solution in the process.
\end{proof}

\section{Representative Selection with Discrete Budgeted Sets}
\label{sec:hardnessdisc}

Finally, we study the complexity of problem (2RS-D), which is the only problem with discrete budgeted uncertainty we consider. Indeed, two-stage problems with discrete budgeted uncertainty are often known to be hard. In this context, the following is an easy adaptation from \cite{GoerigkLendlWulf2022}, which we present for the sake of completeness.

\begin{theorem}
\label{thm:discrete-uncertainty-hard}
The problem (2RS-D) is NP-hard.
\end{theorem}

\begin{proof}
Let an instance of the NP-hard partition problem (see~\cite{karp1972}) be given, consisting of (distinct) values $a_1,\dots,a_n \in \N$.
Let $Q := 1/2 \sum_{i=1}^n a_i$ be half of their total sum. 
The task is to decide if there is a set $I \subseteq [n]$ such that $\sum_{i \in I} a_i = Q$. 
We construct an instance of the two-stage robust representative selection problem with discrete budgeted uncertainty. Let $M > 2Q$ be a large constant, and let $W$ be an even larger constant defined later.
The instance has $2n$ items, $n + 1$ buckets, and $\Gamma=n$. 
The cost of all items is described in Table~\ref{tab:h3}.

\begin{table}
  \begin{center}
    \begin{tabular}{cccccccccc}
      \toprule
      & \multicolumn{3}{c}{$T_1$} & & \multicolumn{1}{c}{$T_2$} & & \multicolumn{1}{c}{$\dots$} & & \multicolumn{1}{c}{$T_{n + 1}$} \\
      \cmidrule{2-4}
      \cmidrule{6-6}
      \cmidrule{8-8}
      \cmidrule{10-10}
      $i$ & 1 & $\dots$ & $n$ & ~ & $n+1$ & ~ & $\dots$ & ~ & $2n$ \\
      \midrule
      $C_i$ & $W$ & $\dots$ & $W$ & & $a_1$ & & $\dots$ & & $a_n$ \\
      $\underline{c}_i$ & $M$ & $\dots$ & $M$ & & 0 & & $\dots$ & & 0 \\
      $\overline{c}_i$ & $M+2Q$ & $\dots$ & $M+2Q$ & & $2a_1$ & & $\dots$ & & $2a_n$ \\
      \bottomrule
    \end{tabular}
  \end{center}
  \caption{Instance constructed in the reduction for the proof of Theorem~\ref{thm:discrete-uncertainty-hard}.\label{tab:h3}}
\end{table}

The bucket $T_1$ is the largest and contains $n$ items, each having 
$C_i=W$, $\underline{c}_i=M$, and $\overline{c}_i = M+2Q$. 
We will choose $W$ as such a large value that packing an item from $T_1$ in the first stage immediately disqualifies a solution from being optimal.
The following analysis shows that $W > M + 4Q$ suffices.
The items in the remaining buckets $T_2,\dots,T_{n+1}$ each correspond to one of the given values $a_i$.
More precisely, for all $i \in \{2,\dots,n+1\}$, the bucket $T_i$ contains a single item
$n + i - 1$ with $C_{n + i - 1} = a_{i - 1}$, $\underline{c}_{n + i - 1} = 0$, and $\overline{c}_{n + i - 1} = 2 a_{i - 1}$.
Note that,
for each of these buckets, the optimal solution must trivially pack its single item. However, there is still the nontrivial choice whether this happens in the first or in the second stage.

Note that an optimal first-stage solution does not select an item in $T_1$.
Therefore, the item from $T_1$ is always selected in the second stage.
The adversary can therefore choose to invest all their budget into $T_1$, causing a cost increase of $2Q$. However, if not all $n$ items of $T_1$ are attacked simultaneously, the attack has no effect.
Hence, the adversary has effectively two valid strategies:
Either the costs of all items in $T_1$ are increased, or the costs of all items in $T_2 \cup \dots \cup T_{n+1}$ are increased. 
In both cases, a second-stage solution will select the remaining items in $T_2 \cup \dots \cup T_{n+1}$ and one item in $T_1$. 
In the first strategy, the cost of the item from $T_1$ is $M+2Q$, and in the second case, it is $M$. 

Let $X$ be the total first-stage cost of the items in $T_2 \cup \dots \cup T_{n+1}$ that are packed by some solution. The total cost of this solution is
\[ X + \max\{ M + 2Q, M + 4Q - 2X \} = M + 2Q + \max\{ X, 2Q - X \}. \]
This is less than or equal to $M+3Q$ if and only if $X=Q$, which can be achieved if and only if the given instance of the partition problem is a yes-instance.
\end{proof}

\begin{remark}
While Theorem~\ref{thm:discrete-uncertainty-hard} implies NP-hardness of (2RS-D),
it is a priori not clear whether the decision variant of (2RS-D) is contained in NP.
However, the containment in NP follows from the following easy argument: 
Note that it suffices to show that the adversarial problem belonging to (2RS-D) can be solved in polynomial time.
Consider a fixed first-stage selection~$x$. For each bucket $T_j$, for $j \in [m]$, we know if there is still an item missing. If this is not the case, the adversary does not increase the cost of any item of this bucket. If an item still needs to be selected in the second stage, it will be an item with minimal cost. This means that the adversary will increase the costs of the items by ordering them from smallest to largest costs~$\underline{c}_i$. For $j \in [m]$, assume that the adversary invests $\Gamma_j \in \fromto{0}{\Gamma}$ of their budget into making items of $T_j$ more expensive.
Let $\alpha_j(\Gamma_j)$ denote the optimal cost that the adversary can cause in bucket~$T_j$ assuming they invest $\Gamma_j$ of their budget. 
The value $\alpha_j(\Gamma_j)$ can therefore be computed in polynomial time.
Given all values $\alpha_j(\Gamma_j)$, the adversary then solves the problem of maximizing $\sum_{j=1}^m \alpha_j(\Gamma_j)$ subject to $\sum_{j=1}^m \Gamma_j = \Gamma$. 
This can be done with a simple dynamic program.
In summary, we conclude that the decision variant of (2RS-D) is NP-complete.
\end{remark}

\section{Conclusions}
\label{sec:conclusions}

Selection problems are fundamental to robust combinatorial optimization due to their simplicity. In many cases, polynomial-time solvability results were first known for problems of this type, and sometimes remain the only such results to this day. Despite having been studied possibly more than any other type of problem in the context of robust combinatorial optimization, our knowledge of the complexity landscape under different types of uncertainty sets has remained fragmental. This paper gives a systematic and thorough understanding of the two-stage problem complexity for the three commonly studied variants of selection problems, in combination with the three commonly studied types of budgeted uncertainty sets. In particular, we settle the long-standing open problem regarding the complexity of two-stage selection with continuous budgeted uncertainty, by showing that this problem is NP-hard. Moreover, we show that the same is true for the two-stage assignment problem with continuous budgeted uncertainty. To the best of our knowledge, these are the first such hardness results for continuous budgeted uncertainty sets, which leads us to expect that similar hardness results for more complex underlying problems can be found in the future.

Our hardness result is based on a reduction from the knapsack problem, so it does not show strong NP-hardness. It remains an open question whether the two-stage selection problem with continuous budgeted uncertainty can be solved in pseudopolynomial time.

There remain special cases to our setting that are yet unexplored and an interesting challenge for further research. If we consider the multi-representative selection problem with the restriction that each $p_j$ is in $\bigO(1)$, we obtain a generalization of the representative selection problem that may be easier to solve than the multi-representative selection problem itself.
We conjecture that, in case of continuous budgeted uncertainty, this problem can be solved in polynomial time by extending the arguments from Section~\ref{sec_poly_alg}.
Furthermore, our hardness results require the budget~$\Gamma$ to be arbitrary, so they do not imply hardness for specific bounds, such as $\Gamma=1$. Finally, it remains open to explore recoverable problem variants in the same way as two-stage problems have been studied in this paper.

\paragraph*{Acknowledgements.} Lasse Wulf was supported by Eva Rotenberg's Carlsberg Foundation Young Researcher Fellowship CF21-0302 ``Graph Algorithms with Geometric Applications''.

\printbibliography

\end{document}